\documentclass[12pt,a4paper]{amsart}

\usepackage[a4paper]{geometry}
\usepackage{amsmath,amsthm,amsfonts,amssymb,mathrsfs,amsbsy}

\usepackage{xcolor}
\usepackage{hyperref}
\usepackage{mathtools}
\mathtoolsset{showonlyrefs}
\usepackage{algpseudocode}
\usepackage{algorithm}

\newtheorem{theorem}{Theorem}

\newtheorem{lemma}{Lemma}
\newtheorem*{lemma*}{Lemma}
\newtheorem{remark}{Remark}
\newtheorem*{remark*}{Remark}
\newtheorem*{theorem*}{Theorem}

\numberwithin{equation}{section}
\numberwithin{theorem}{section}
\numberwithin{lemma}{section}

\title[The regional control of a fractional spatio-temporal SIR model]{The regional control of a fractional spatio-temporal SIR model}

\author[Sofwah Ahmad]{Sofwah Ahmad} 
\address[Sofwah Ahmad]{
Mohamed bin Zayed University of Artificial Intelligence, Abu Dhabi, UAE\\ 
\href{https://orcid.org/0000-0001-7641-7759}{orcid.org/0000-0001-7641-7759}}
\email{sofwah.ahmad@mbzuai.ac.ae, alaydrus.sofwah@gmail.com}

\begin{document}

    \begin{abstract}
        This paper investigates a regional optimal control problem for a nonlinear spatio-temporal epidemiological model involving fractional diffusion on a bounded domain. In this work, we consider a general form of disease transmission, and two types of control, vaccination and treatment. We establish the existence and uniqueness of global-in-time solutions to our proposed system. We also prove the existence of an optimal control that minimizes the infected individuals as well as the cost of vaccination and treatment. The necessary conditions for optimality are derived. We show that concentrating vaccination in a selected region of the domain can substantially mitigate the spread of the disease. Numerical simulations are given.
        
        \medskip
    
        
        
    
    \end{abstract}

\maketitle

    \section{{Introduction}}
In this paper, we are interested to study the regional optimal control strategy of the following system 
        \begin{align}
            \frac{\partial S}{\partial t}  &= -d_1(-\Delta_N)^{\sigma} S - SI\ \psi(I)+ f_1(S,I) - \chi_{\omega}(x) u_1 S,
            \label{Eq : FDE_control_system1}
            \\
            \frac{\partial I}{\partial t} &= -d_2(-\Delta_N)^{\sigma}  I + SI\ \psi(I)+ f_2(S,I) - u_2 I,  
            \label{Eq : FDE_control_system2}
        \end{align}
which describes the propagation of an infectious disease within a large population; $S(t,x)$ and $I(t,x)$ denote the densities of susceptibles and infected compartment at time $t>0$ and spatial position $x\in\Omega$, subject to initial conditions
        \begin{align}
            S(0,x)&=S_0, 
            \quad I(0,x)=I_0, 
            \qquad x\in\Omega,\label{Eq: IC_system}
        \end{align}
with $\Omega\subset \mathbb{R}^2$ bounded,  $u_1(t,x), u_2(t,x)$ are the controls (that denotes the vaccine and the treatment, respectively) whose space of functions will be specified later, and $\chi_\omega(x)$ is the characteristic function of $\omega\subset\Omega$, defined as  
    \begin{align}
        \chi_\omega(x) =
            \begin{cases}
            1, & x \in \omega, \\
            0, & x \notin \omega.
            \end{cases}
    \end{align}



We notice that when $u_1=u_2=d_1=d_2=0$, $f_1(S,I)=0$, $f_2(S,I)=-\gamma I$ and $\psi(I)=\beta$ with $\gamma,\beta\in\mathbb{R}^+$, the model is reduced to the well-known Kermack and Mc-Kendrick model {\cite{Kermack1927}} with additional equation in the system
    $$\frac{dR}{dt}=\gamma I,$$
where $R(t)$ represents the removed compartment (\textit{i.e} the indiviuals who are either immune, recovered or passed away). However, to study the model, we need to focus only on the first two equations as the $R$ compartment can be studied once the results for $S$ and $I$ are obtained.

More variations of SIR model has been developed to take into account many factors such as natural birth and death rates as well as the type of interaction that gives rise to different power of the nonlinearities leading to more extensive expression in the right-hand side of \eqref{Eq : FDE_control_system1}-\eqref{Eq : FDE_control_system2}.

For example, if we consider the SIRS model by Lahrouz et. al \cite{LAHROUZ20126519} but omitting the possibilities of reinfection, then we have the system \eqref{Eq : FDE_control_system1}-\eqref{Eq : FDE_control_system2} with $d_1=d_2=u_1=u_2=0$, and 
    $$f_1(S,I) = (1-p) b-\mu S, \quad
    f_2(S,I)= -(\mu+c+\alpha) I,\quad
    \psi(I)=\frac{\beta}{\phi(I)}.$$
Here, $b$ denotes the new individuals entering the population, $p\in[0,1]$ denotes the portion of individuals who are vaccinated, hence $(1-p)$ denote the individuals who are susceptible, $\mu$ is the natural death rate, $c$ is the death rate caused by the disease, and $\alpha$ is the recovery rate.  In this model, the transmission of the disease is governed by $\frac{\beta S I}{\phi(I)}$ where the authors assume $\phi$ to be a positive function such that $\phi(0)=1$ and $\phi'>0$.

However, since mobility between individuals in the population certainly contributes to the rate of contact between individuals and therefore to the spread of the disease, it is necessary to take into account the spatial distribution in the model. This can be done by incorporating the diffusion terms with $d_1\ne 0$ and $d_2\ne 0$ in the system \eqref{Eq : FDE_control_system1}-\eqref{Eq : FDE_control_system2}. In this work, we consider the fractional power $\sigma\in (0,1]$ (that shall be precised later) to take into account the possibility of certain individuals with high mobility to jump to remote sites.
We notice that the classical spatio-temporal case where $\sigma=1$ has been well studied in the literature before for epidemiological study and reaction diffusion system in general.
Capasso \cite{Capasso1978} studied the system with $d_1\ne d_2$, $u_1=u_2=0$, $f_1=0$, $f_2=-\lambda I$, $\lambda>0$, with the interaction term being $SI\psi(I)=\Psi(I)S$ where $\Psi(I)$ is smooth bounded nonlinear function. The boundary conditions considered is of Neumann type;
using semigroup theory, he showed the existence of a strong global solution. 
Webb \cite{WEBB1981} studied the similar system as Capasso{\cite{Capasso1978}} with $\psi(I)=a>0$ equipped with Dirichlet boundary conditions. Using semigroup theory and Lyapunov method, he proved the existence of solutions in spatially inhomogeneous case, and prove the convergence of $S(x,t)$ to $S_\infty>0$ and $I(x,t)$ to $0$ as $t\to +\infty$. 
We need to highlight that later, Haraux and Kirane \cite{HarauxKirane1983} studied a more general version of model of \cite{WEBB1981} and established the uniform boundedness of solution in $C^1(\Omega)$.


In this work, rather focusing on modelling the disease progression, we aim to study the model that control the disease spread by adding the vaccination $u_1(x,t)$ and treatment (or medication) $u_2(x,t)$. Since national resources are typically limited during a pandemic, our goal is to control the spread of the disease by minimizing the number of infections while keeping vaccination and treatment costs as low as possible. Moreover, instead of adopting a nationwide vaccination strategy as was commonly implemented during the COVID-19 pandemic, where nearly the entire population was vaccinated, we propose a regional control approach, that is vaccinating only selected areas to effectively manage the disease outbreak, reduce the national burden, and limit potential side effects.



We notice the optimal control problem applied to spatio-temporal SIR model and its variations has been studied before. For example, 
Laaroussi et al. \cite{Laaroussi2018} considered an optimal vaccination control problem to a spatiotemporal SIR model to minimize the densitiy of the infected individuals and the cost for vaccination. 
Zhou et al \cite{Zhou2024} studied the optimal control problem for a reaction-diffusion SIRC model with cross-immunization and incorporated two control strategies into their proposed epidemic system. 
Ammi et al. \cite{SIDIAMMI2023} considered an optimal vaccination control problem to a spatiotemporal SIR model involving $p-$Laplacian.
Using semigroup theory and appropriate embedding, the authors established the existence and uniqueness of strong positive global solutions and established the optimal control strategy. 
Then, Laaroussi and Rachik \cite{Laaroussi2020} studied a regional optimal control problem of a reaction diffusion SIR model, which was first studied deeply by Zerrik, Boutoulout and El Jai \cite{zerrik}. They establish the necessary optimality condition for the optimal control problem.

Up until the writing of this this paper, to the best of the author knowledge, there is no existing literature known on a regional optimal control problem, applied to nonlinear model with fractional Laplacian. A study of control theory in linear reaction-diffusion systems involving fractional Laplacian has been studied before (\textit{e.g} in \cite{biccari2021}). However, as mentioned in \cite{biccari2021}, controlling the nonlinear dynamical system is rendered to be one of the most challenging open problem, that is still minimally explored. For this reason, this paper would like to answer this challenge and shall attempt to give interpretation through the simulation of our findings. 



    \subsection*{Model assumptions}

    For the well-posedness analysis developed in the sequel, we impose the following assumptions on the model \eqref{Eq : FDE_control_system1}--\eqref{Eq : FDE_control_system2}. Let  $Q_T=[0,T]\times\Omega$ and $Q_{\omega}=[0,T]\times\omega$.

        \begin{enumerate}
            \item The functions $f_1(S,I),f_2(S,I)$ are of the form 
                \begin{align*}
                    f_1(S,I)&= -a_1(x)S+b(x);\\
                    f_2(S,I)&= -a_2(x)I,
                \end{align*}
            where $a_1(x)$  death rate of natural causes for the susceptible class; $a_2(x)$ represents the total removal rate from the infected compartment, which consists of the recovery rate, the disease-induced death rate, and the natural death rate; and $b(x)$ represents the birth rate of the susceptible class. For biological reasons, we assume that $a_1, a_2, b\in L^\infty(\Omega)$ and are positive.
            \item $S(t,x)+I(t,x)\le N(t,x)$, where $N(t,x)$ denotes the total population density at time $t$ and spatial position $x$.
            \item The function $\psi$ satisfies $\psi\in C^1(\mathbb{R}^+)$.
        \end{enumerate}
        
        Under these assumptions, our main objective is to minimize the $I$ compartment and to reduce the necessary cost of control $u_1$ and $u_2$, that can be achieved by minimizing the following objective functional
                \begin{align}
                    J(I,u) =& \frac{\rho}{2}\|I\|^2_{L^2(Q_T)}
                    + \frac{\sigma_1}{2} \|u_1\|^2_{L^2(Q_\omega)}
                    + \frac{\sigma_2}{2} \|u_2\|^2_{L^2(Q_T)},
                    \label{Eq: Objective_Functional}
                \end{align}      
        where $\rho, \sigma_1,\sigma_2$ are weight parameters and $u_1(x,t)>0$ and $u_2(x,t)>0$ belong to the set of admissible controls 
            \begin{align}
                U_{ad}= \{(u_1,u_2)\in L^{\infty}(Q_{\omega})\times L^{\infty}(Q_T) \ : 0\le u_1\leq U_1 \text{ and } 0\le u_2 \le U_2 \ \ a.e\}
                \label{Eq: Uadmissible system 2}
            \end{align}
            
       The main objective will be achieved by characterizing the optimal control $u^*\in U_{ad}$ that minimizes \eqref{Eq: Objective_Functional}. 
       In Section \ref{sec:con_Pre} we recall the definitions and lemmas that will be used throughout the paper. In Section \ref{sec:con_existence}, we establish the existence of a unique positive and bounded strong global solution to problem \eqref{Eq : FDE_control_system1}-\eqref{Eq : FDE_control_system2} equipped with the initial condition \eqref{Eq: IC_system}. Then in Section \ref{sec:con_existence_op}, we establish the existence of the optimal solution for the control problem \eqref{Eq: Objective_Functional} subject to the fractional system \eqref{Eq : FDE_control_system1}-\eqref{Eq : FDE_control_system2}. We discuss the necessary optimality conditions ofthe problem \eqref{Eq : FDE_control_system1}-\eqref{Eq: Uadmissible system 2} in Section \ref{sec:con_necessary_op}.\\

    \section{{Preliminaries}}
    \label{sec:con_Pre}
    In this section, we precise the definition of the fractional laplacian $(-\Delta_N)^\sigma$ used in this paper, and some of the tools that will be needed in our discussion.
    We start by recalling the spectral definition of the fractional Laplace operatorassociated with the Neumann Laplacian. To this end, let $\Omega \subset \mathbb{R}^d$
    be a bounded open set with sufficiently smooth boundary. We denote by
    $\{\lambda_n\}_{n=0}^{\infty}$ the sequence of eigenvalues of the operator $-\Delta$ in $\Omega$ subject to the homogeneous Neumann boundary condition, with
        \begin{align*}
            0=\lambda_0 < \lambda_1 \leq \lambda_2 \leq \cdots,
            \qquad \lambda_n \to +\infty \quad \text{as } n\to\infty.
        \end{align*}
    Let $\{e_n\}_{n=0}^{\infty}$ be the corresponding normalized eigenfunctions in
    $L^2(\Omega)$, that is,
        \begin{align*}
            -\Delta e_n &= \lambda_n e_n, \quad &&x\in \Omega, \\
            \frac{\partial e_n}{\partial \eta} &= 0, \quad &&x\in \partial\Omega,
        \end{align*}
    where $\eta$ denotes the outward unit normal vector on $\partial\Omega$. 

    The family $\{e_n\}_{n=0}^\infty$ forms a complete orthonormal basis of $L^2(\Omega)$.
    
    For $0<s\le 1$, we define the set $H^s(\Omega)$ as
    \begin{align*}
        H^s(\Omega)
        :=
        \left\{
            u=\sum_{n=0}^{\infty} u_n e_n \in L^2(\Omega)
            :
            \sum_{n=0}^{\infty} \lambda_n^s |u_n|^2 <+\infty
        \right\},
    \end{align*}
    where
        \begin{align*}
            u_n = \int_{\Omega} u(x)e_n(x)\,dx.
        \end{align*}

    The spectral fractional Laplace operator $(-\Delta_N)^s$ is then defined by
        \begin{align*}
            (-\Delta_N)^s u
            :=
            \sum_{n=0}^{\infty} \lambda_n^s u_n e_n,
        \end{align*}
    with domain
        \begin{align*}
            D\big((-\Delta_N)^s\big)
            :=
            \left\{
                u=\sum_{n=0}^{\infty} u_n e_n \in L^2(\Omega)
                :
                \sum_{n=0}^{\infty} \lambda_n^{2s}|u_n|^2 <+\infty
            \right\}.
        \end{align*}
    In particular, since $\lambda_0=0$, constant functions belong to the kernel of $(-\Delta_N)^s$.

    To prove our result, we shall need some lemmas. First, using the properties of scalar products, we immediately have the following integration by parts formula for the fractional Laplacian.

    \begin{lemma}[Integration by parts]
        For $0<\sigma\le 1$ and $u,v\in D\left((-\Delta_N)^\sigma\right)$, the following identity holds 
            \begin{align}
                \int_{\Omega} u (-\Delta_N)^\sigma v \ dx =  \int_{\Omega} v (-\Delta_N)^\sigma u \ dx.
            \label{eq: int by parts lap}
            \end{align}
    \end{lemma}
    We also need the following inequality of Stroock and Varopoulos.
        \begin{lemma}\cite[Theorem 1]{LiskevichSemenov}
            Let $0<\sigma\le 2$ and $p> 1$. For nonnegative $u\in L^p(\Omega)$ with $(-\Delta_N)^{\frac{\sigma}{2}} u \in L^p(\Omega)$ we have
                \begin{align*}
                    \int_{\Omega} u^{p-1} (-\Delta_N)^{\frac{\sigma}{2}}u \ dx 
                    \ge \frac{4(p-1)}{p^2} \int_{\Omega} \left|(-\Delta_N)^{\frac{\sigma}{4}} u^{\frac{p}{2}}\right|^2 dx. 
                \end{align*}
            \label{Lemma_Strook&Varopulous_ineq}
        \end{lemma}
    
    \begin{lemma}\cite[Proposition 1.2]{Barbu1994}
        Let $X$ be a real Banach space, $A: D(A)\subseteq X \longrightarrow X$ be the infinitesimal generator of a $C_0-$semigroup of contraction $\mathcal{S}(t)$, $t>0$ on $X$ and $f: [0,T]\times X \longrightarrow X$ be a function measurable in $t$ and Lipschitz continuous in $x\in X$.
        \begin{enumerate}
            \item [(i)] (Local Existence) If $y_0\in X,$ then the problem \eqref{Eq: Abstract_FDE_control} admits a unique mild solution, i.e, a function $y\in C([0,T]; X)$ that satisfies
                \begin{align}
                    y(t) = \mathcal{S}(t)y_0 + \int_0^t \mathcal{S}(t-s) f(s,y(s))\ ds, \quad\forall t\in[0,T]
                    \label{eq:control_int_rep}
                \end{align}
            \item[(ii)] (Global Existence) If $X$ is a Hilbert space, $A$ is self-adjoint and dissipative on $X$, and $y_0\in D(A)$, then the mild solution is in fact a strong solution, i.e $y$ satisfies \eqref{eq:control_int_rep} and $y\in W^{1,2}([0,T];X)\cap L^2(0,T;D(A))$.
        \end{enumerate}
        \label{lemma : Local and Global Existence}
    \end{lemma}

    \section{{Existence of strong positive global solution}}
    \label{sec:con_existence}
    Let $y=(y_1,y_2)=(S,I)$, $y(0)=y_0=(y_{10},y_{20})=(S_0,I_0)$, $D_A(\Omega)=L^2(\Omega)\times L^2(\Omega)$. Let $A$ be the linear operator defined by 
            \begin{align*}
                A: D(A)\subset D_A(\Omega) &\longrightarrow D_A(\Omega)\\
                y & \longmapsto 
                Ay = \left(-d_1(-\Delta_N)^{\sigma} y_1,-d_2(-\Delta_N)^{\sigma} y_2 \right),
            \end{align*}
        with domain of $A$ is given by 
        $D(A) := D\big((-\Delta_N)^\sigma\big) \times D\big((-\Delta_N)^\sigma\big).$
        
        Consider $F(t,x,y(t))=\left(F_1(t,x,y(t)),F_2(t,x,y(t))\right)$ where
            \begin{align}
                \begin{cases}
                    F_1(t,x,y) &= -y_1y_2\psi(y_2) -a_1 y_1 + b - \chi_{\omega}u_1 y_1, \\
                    F_2(t,x,y) &= y_1y_2\psi(y_2) -a_2 y_2 - u_2 y_2,
                \end{cases}
                \label{eq : F1 F2 RHS}
            \end{align}
        with $y_{10}+y_{20}\le N_0 :=\ \sup_{x\in\Omega} N(0,x) \ <\ \infty,$for $x \in \Omega,$ a.e. Then the problem \eqref{Eq : FDE_control_system1}-\eqref{Eq: IC_system} can be rewritten in the space $D_A(\Omega)$ in the form
            \begin{align}
                \begin{cases}
                    \frac{dy}{dt} = Ay + F(t,y(t)), \qquad t\in [0,T],\\
                     y(0) = y_{0}.
                \end{cases}
            \label{Eq: Abstract_FDE_control}
            \end{align}
        
        Furthermore, let 
        $W^{1,2}\left([0,T],D_A(\Omega)\right)=
        \left\{y\in AC([0,T],D_A(\Omega)): \frac{\partial y}{\partial t}\in L^2(0,T; D_A(\Omega))\right\}$ 
        and 
        
        \noindent 
        $N(T,\Omega) := L^2\big(0,T; D((-\Delta_N)^\sigma)\big) \cap L^\infty\big([0,T]; H^\sigma(\Omega)\big).$

        For $T>0$, consider the Banach space $K_T=C\left([0,T],C(\bar{\Omega})\times C(\bar{\Omega}) \right)$ equipped with the norm
            \begin{align*}
                \|y\|_{K_T} := \sup_{t\in[0,T]}\Big(\|y_1(t)\|_{L^\infty(\Omega)}+\|y_2(t)\|_{L^\infty(\Omega)}\Big),
                \qquad\text{for any } y=(y_1,y_2)\in K_T.
            \end{align*}
        First, we prove the following lemma.

        \begin{lemma}
            The function $F(t,y)=(F_1(t,y_1,y_2), F_2(t,y_1,y_2))$ is Lipschitz continuous on $C\left([0,T],C(\bar{\Omega})\times C(\bar{\Omega}) \right)$ with respect to y.
        \label{lemma F Lipschitz}
        \end{lemma}
        \begin{proof}   
            We show that there exist $L>0$ such that 
                \begin{align*}
                    \|F(t,y)-F(t,\tilde{y})\|_{\infty} < L\ \|y-\tilde{y}\|_{\infty}
                \end{align*}
                
            Notice that, for any $y,\tilde{y}\in K_T$, there exist $M>0$ such that 
                \begin{align*}
                    \|y\|_\infty < M,
                    \quad |y_i|<M,
                    \quad |\psi(y_2)|<M,
                    \qquad\text{for i = 1,2.}
                \end{align*}
            Hence,
                \begin{align*}
                    |F_1(t,y)-F_1(t,\tilde{y})|
                    \le|y_1y_2\psi(y_2)-\tilde{y}_1\tilde{y}_2\psi(\tilde{y}_2)|
                    + |a_1(x)| |y_1-\tilde{y}_1|
                    + |\chi_{\omega} u_1|\ |y_1-\tilde{y}_1|.
                \end{align*}
            We notice that $\psi(y_2(t,x))$ is Lipschitz continuous on $C([0,T], C(\Omega))$, so there exist Lipschitz constant $L'$ such that 
                \begin{align*}
                    |y_1y_2\psi(y_2)-\tilde{y}_1\tilde{y}_2\psi(\tilde{y}_2)|
                    &\le
                    |y_1-\tilde{y}_1|\ |y_2\psi(\tilde{y}_2)| 
                    + |y_1| |(y_2-\tilde{y}_2)\psi(\tilde{y}_2)
                     +y_2(\psi(y_2)-\psi(\tilde{y}_2))|
                    \\
                    &\le
                    M^2 |y_1-\tilde{y}_1|\ 
                    + M^2 |y_2-\tilde{y}_2|\
                    + M^2 \ C\ |y_2-\tilde{y}_2|
                    \\
                    &\le L'\left(|y_1-\tilde{y}_1|+|y_2-\tilde{y}_2|\right),
                    \intertext{so}
                    |F_1(t,y)-F_1(t,\tilde{y})|
                    &\le
                      L'\left(|y_1-\tilde{y}_1|+|y_2-\tilde{y}_2|\right)
                    + |a_1(x)| |y_1-\tilde{y}_1|
                    + |\chi_{\omega} u_1|\ |y_1-\tilde{y}_1|\\
                    &\le
                    L_1\left(|y_1-\tilde{y}_1|+|y_2-\tilde{y}_2|\right),
                    \intertext{ with 
                    $L_1 = \sup_{x\in\Omega}\{L' + |a_1(x)| + U_1\},$.
                    Similarly,}
                     |F_2(t,y)-F_2(t,\tilde{y})| 
                     &\le L_2\left(|y_1-\tilde{y}_1|+|y_2-\tilde{y}_2|\right),
                     \intertext{where 
                     $L_2 = \sup_{x\in\Omega}\{L' + |a_2(x)| + U_2\}$. Hence,}
                     \|F(t,y)-F(t,\tilde{y})\|_{\infty} 
                     &\le \|F_1(t,y)-F_1(t,\tilde{y})\|_\infty + \|F_2(t,y)-F_2(t,\tilde{y})\|_\infty
                     \\
                     &\le L_1 \|y-\tilde{y}\|_\infty + L_2 \|y-\tilde{y}\|_\infty
                     \\ 
                     &=L_3 \|y-\tilde{y}\|_\infty , \qquad\qquad \text{where } L_3:=L_1+L_2 .
                \end{align*}
                
        \end{proof}

    
    Notice that $D(A)$ is dense in $D_A(\Omega)$, and $A$ is self-adjoint and dissipative on $D(A)$; indeed, for all $y=(y_1,y_2)\in D(A)$,
        \[
        \langle Ay, y\rangle_{D_A(\Omega)}
        = -d_1 \big\|(-\Delta_N)^{\sigma/2} y_1\big\|^2_{L^2(\Omega)}
          -d_2 \big\|(-\Delta_N)^{\sigma/2} y_2\big\|^2_{L^2(\Omega)}
        \le 0.
        \]
    Hence, by the Lumer--Phillips theorem, $A$ is the infinitesimal generator of a $C_0$-semigroup of contractions $\mathcal{S}(t)$, $t>0$, on $D_A(\Omega)$. 
 
    \begin{theorem}(Local Existence)
        Let $ y_{10}, y_{20}\in C(\Omega)$ with $ y_{10}, y_{20}\ge 0$. Then, there exist  $T>0$ and a unique mild solution $y\in K_T$ that satisfies
            \begin{align}
                y(t) = \mathcal{S}(t)y_0 + \int_0^t \mathcal{S}(t-s) f(s,y(s))\ ds, \quad\forall t\in[0,T],
            \end{align}
        where $\mathcal{S}(t)$ is the $C_0-$semigroup of contraction generated by $A$. 
    \end{theorem}

    Next we show the positivity and boundedness of the solution.

    \begin{lemma}     
        Let $( y_{10},  y_{20})\in C(\Omega)\times C(\Omega)$ be such that $ y_{i0}\ge 0$ for $i=1,2$. Then $y_i(t,x)\ge 0$ and $y_i\in\ L^\infty(Q_T)$.
    \end{lemma}
    \begin{proof}
        Let $y_i= y_i^+ + y_i^-,$ with $y_i^+ = \sup(y_i,0)$ and $y_i^- = \inf(y_i,0)$ for $i=1,2$.
        Multiplying the first equation of \eqref{Eq: Abstract_FDE_control} by $y_1^-$ and integrating with respect to $x$ over $\Omega$, we get 
            \begin{align*}
                \int_{\Omega} y_1^- \frac{\partial y_1^-}{\partial t} \ dx
                &\le -d_1 \int_{\Omega} y_1^-(-\Delta_N)^\sigma y_1^- \ dx
                  -\int_{\Omega} (y_1^-)^2 y_2\psi(y_2) \ dx
                  -\int_{\Omega}a_1(x) (y_1^-)^2 \ dx
                  \\
                  & \quad
                  +\int_{\Omega}b(x) y_1^-  \ dx
                  -\int_{\Omega}\chi_\omega u_1 (y_1^-)^2 \ dx.
                \intertext{Using integration by parts, we obtain,}
                \frac{1}{2} \frac{\partial }{\partial t}\int_{\Omega} (y_1^- )^2 \ dx
                &= -d_1 \int_{\Omega} |(-\Delta_N)^{\sigma/2} (y_1^-)|^2\ dx
                  -\int_{\Omega} (y_1^-)^2 y_2\psi(y_2) \ dx
                  -\int_{\Omega}a_1(x) (y_1^-)^2 \ dx
                  \\
                  & \quad
                  +\int_{\Omega}b(x) y_1^-  \ dx
                  -\int_{\Omega}\chi_\omega u_1 (y_1^-)^2 \ dx\\
                &\le -\int_{\Omega}(y_1^-)^2y_2\psi(y_2) \ dx.
            \end{align*}
        From the local existence, $y_2$ is bounded on $[0,T]$ for $T<T_{\max}$; therefore, there exist a continuous function $C_2(t)$ (depending on $t$), such that $\|y_2(t)\|_{L^\infty(\Omega)}\le C_2(t)$ and $\|\psi(y_2(t))\|_{L^\infty(\Omega)}\le MC_2(t)$. Hence, we have
            \begin{align*}
                \frac{\partial }{\partial t}\int_{\Omega} (y_1^- )^2 \ dx
                &\le 
                   M C_2(t)^2\int_{\Omega} (y_1^-)^2 \ dx.
            \end{align*}
        Let $\Phi_1(t) = \int_{\Omega}(y_1^-)^2 \ dx$. Since $ y_{10}= y_{10}^+$ then $ y_{10}^-=0$ so that 
        $\Phi_1(0) 
        =0$. Then for $C(t)=M(C_2(t))^2$ we have
            \begin{align*}
                \begin{cases}
                    \frac{\partial }{\partial t} \Phi_1 (t) \le C(t) \Phi_1(t), 
                    \\
                    \Phi_1(0)=0.
                \end{cases}
            \end{align*}
        By Gronwall's lemma we deduce that 
            \begin{align*}
                \Phi_1(0) = 0 
                \quad\Longrightarrow\quad  
                \|y_1^-\|_{L^2(\Omega)}=0
                \quad\Longrightarrow\quad 
                y_1^- = 0 \ a.e
                \quad\Longrightarrow\quad 
                y_1 = y_1^+\ge 0.
            \end{align*}
        Next, by multiplying the second equation of \eqref{Eq: Abstract_FDE_control} by $y_2^-$ and integrating over $\Omega$,
        using similar reasoning as before, we deduce that $y_2 \ge 0$.\\

         \noindent\textbf{Boundedness}. 
        To prove the boundedness of $y_1$, we take the scalar product of the first equation of \eqref{Eq: Abstract_FDE_control} by $y_1^{p-1}$, with $p>1,$ to get 
            \begin{align*}
                \int_{\Omega} y_1^{p-1} \frac{\partial y_1}{\partial t} \ dx
                &= -d_1 \int_{\Omega} y_1^{p-1}(-\Delta_N)^\sigma y_1 \ dx
                  -\int_{\Omega}\left( y_1^{p}y_2\psi(y_2) 
                  +a_1(x) y_1^{p} \right) \ dx
                  \\
                  &\quad
                  +\int_{\Omega}b(x) y_1^{p-1}\ dx 
                  -\int_{\Omega}\chi_\omega u_1 y_1^p dx.
                  \\
                \frac{1}{p} \frac{\partial }{\partial t}\int_{\Omega} y_1^p \ dx
                &\le -d_1 \int_{\Omega} y_1^{p-1}(-\Delta_N)^\sigma y_1 \ dx
                  +\int_{\Omega}b(x) y_1^{p-1}\ dx   
                  \le
                \end{align*}
        Using Stroock and Varopoulos inequality \eqref{Lemma_Strook&Varopulous_ineq},  we obtain
            \begin{align*}
                \frac{1}{p} \frac{\partial }{\partial t}\|y_1\|_{L^p(\Omega)}^p \ dx
                &\le -\frac{4d_1(p-1)}{p} \int_{\Omega} \left|(-\Delta_N)^{\sigma/2} (y_1^{p/2})\right|^2\ dx
                  +b\int_{\Omega} y_1^{p-1}\ dx .
                \intertext{We notice that, by choosing $1/r+1/s=1$ with $s=p$ and $r=p/(p-1)$ then}
                \int_{\Omega} y_1^{p-1}
                &\le \left(\int_\Omega y^{(p-1)r} \ dx\right)^{1/r} \left(\int_\Omega 1^s \ dx\right)^{1/s} 
                = \left(\int_\Omega y^{p} \ dx\right)^{1-1/p} |\Omega|^{1/p}
                =\|y_1\|_{L^p(\Omega)}^{p-1}|\Omega|^{1/p}
            \end{align*}
         Hence,
            \begin{align*}
                \frac{1}{p}p \|y_1\|_{L^p(\Omega)}^{p-1} \frac{\partial}{\partial t }\|y_1\|_{L^p(\Omega)}
                &\le b \ \|y_1\|_{L^p(\Omega)}^{p-1}|\Omega|^{1/p}
                \intertext{or,}
                 \frac{\partial}{\partial t }\|y_1\|_{L^p(\Omega)}
                &\le b \ |\Omega|^{1/p}
            \end{align*}
        or
            \begin{align*}    
                \|y_1\|_{L^p(\Omega)} &\le  \|y_{10}\|_{L^p(\Omega)} + b|\Omega|T.
            \end{align*}
        By taking $p\to \infty$, we get $y_1\in L^\infty(\Omega)$. Hence, we deduced that $y_1\in L^\infty(Q_T)$

        To find the bound for $y_2$ we shall use duality argument (see \cite{MichelePierre2010}). Adding both equations in the system \eqref{Eq: Abstract_FDE_control} and using the positivity $y_1,y_2$ we get 
            \begin{align}
                \partial_t y_2+ d_2(-\Delta)^\sigma y_2\le -\partial_t y_1 -d_1(-\Delta)^\sigma y_1+b
                \label{ineq:sum dual}
            \end{align}

        Let $\phi$ be the solution of  
            \begin{align}
                \begin{cases}
                    \phi_t + d_2 (-\Delta_N)^\sigma \phi = -\psi, & \text{in } (0,T) \times \Omega, \\
                    \phi(0) = \phi_0, & x \in \Omega.
                \end{cases}
                \label{sys:dual bddness}
            \end{align}
        with $\psi \in C_0^\infty(0,T;\Omega)$, $\psi > 0$ and $\phi_0 \in B^{2\sigma(1-1/q)}_{q,q}(\Omega)\cap L^\infty(\Omega)$, 
        where $B^s_{p,q}(\Omega)$ denotes the Besov space.
        
        \begin{remark}
            The Besov space $B^s_{p,q}(\Omega)$ is the space consisting of all functions $f \in W^{m,p}(\Omega)$ (with $m = \lfloor s \rfloor$) for which the seminorm
            \begin{equation}
                [f]_{B^s_{p,q}(\Omega)} := 
                \left( \sum_{|\alpha| = m} \int_0^1 \left( t^{-s + m} \sup_{|\tilde{h}| \le t} \| \Delta_{\tilde{h}}^2 D^\alpha f \|_{L^p(\Omega_{2\tilde{h}})} \right)^q \frac{dt}{t} \right)^{1/q}
            \end{equation}
        is finite, and the norm is given by
            \begin{equation}
                \|f\|_{B^s_{p,q}(\Omega)} := \|f\|_{W^{m,p}(\Omega)} + [f]_{B^s_{p,q}(\Omega)},
            \end{equation}
        where $\Delta_{\tilde{h}}^2 g(x) = g(x+\tilde{h}) - 2g(x) + g(x-\tilde{h})$ and $\Omega_{2\tilde{h}} := \{ x \in \Omega : x \pm \tilde{h} \in \Omega \}$.
        (see \textit{e.g} \cite[p. 8]{Triebel1992}).
        \end{remark}
        
        \begin{remark}
            Proposition 3.3 in the work of Denk, Hieber and Pr\"uss \cite{DenkHieberPruss2007} ensures the existence and uniqueness of the solution of \eqref{sys:dual bddness} 
                \begin{align}
                    \phi \in H^1(0,T; L^q(\Omega)) \cap L^q(0,T; D((-\Delta_N)^\sigma)).
                \end{align}
            Furthermore, for $p\in(1,\infty)$ and $q=p/(p-1)$ the solution satisfies the maximal $L^q$-regularity estimate:
                \begin{align}
                    \|\phi_t\|_{L^q((0,T) \times \Omega)} + \|(-\Delta_N)^\sigma \phi\|_{L^q((0,T) \times \Omega)} 
                    \leq C \|\psi\|_{L^q((0,T) \times \Omega)},
                    \label{ineq: maximal regularity}
                \end{align}
        \end{remark}

        \noindent 
        Multiplying \eqref{ineq:sum dual} with $\phi$ and integrate over $Q_T$ gives
            \begin{align*}
                 \int_{Q_T}\phi \left(\partial_t y_2 + d_2(-\Delta)^\sigma y_2\right)\ dxdt
                 &\le -\int_{Q_T} \phi \left(\partial_t y_1 + d_1(-\Delta)^\sigma y_1\right)\ dxdt+\int_{Q_T}\phi b \ dxdt.
                 \end{align*}
        Integration by parts formula \eqref{eq: int by parts lap} yields
            \begin{align*}
                \int_{Q_T}\left(-\partial_t \phi + d_2(-\Delta)^\sigma \phi\right)y_2\ dxdt
                 &\le \int_{\Omega}\phi_0(y_{20}-y_{10}) \ dx+\int_{Q_T} \left(\partial_t \phi - d_1(-\Delta)^\sigma\phi +\phi b \right)y_1\ dxdt,
            \end{align*}
        which is equivalent to
            \begin{align*}
                \int_{Q_T} \psi\ y_2\ dxdt
                 &\le C+\int_{Q_T} \left(\partial_t \phi - d_1(-\Delta)^\sigma\phi +\phi b \right)y_1\ dxdt.
            \end{align*}
        for some contant $C>0$. Using the estimate \eqref{ineq: maximal regularity}, we have 
            \begin{align*}
                \int_{Q_T} \psi\ y_2\ dxdt
                 &\le C+ \left(\|\partial_t \phi\|_{L^q(Q_T)} + d_1 \|(-\Delta)^\sigma\phi\|_{L^q(Q_T)} + b \|\phi\|_{L^q(Q_T)} \right)\|y_1\|_{L^\infty(Q_T)}\\
                 &\le C(1+\|\psi\|_{L^q(Q_T)})<+\infty
            \end{align*}
        Hence, by duality we deduce that
            \begin{align*}
                \|y_2\|_{L^p(Q_T)} < C.
            \end{align*}
        Taking $p\to \infty$ we have
            \begin{align*}
                \|y_2\|_{L^\infty(Q_T)} < C.
            \end{align*}
    \end{proof}   

    \begin{theorem}
        For $u\in U_{ad}$, $y_0\in D(A)$, and $ y_{i0}\ge 0$ on $\Omega$ for $i=1,2$, the problem \eqref{Eq: Abstract_FDE_control} has a unique strong global solution $y\in W^{1,2}(0,T; D_A(\Omega))$ such that $y_i\in N(T,\Omega)\cap L^\infty(Q_T),$ for $i=1,2$. Moreover, there exist $C>0$ such that for any $t\in [0,T]$ the following inequality holds
             \begin{align}
               \left\|\frac{\partial y_i}{\partial t}\right\|_{L^2(Q_T)}
               + \|y_i\|_{L^2(0,T;D((-\Delta_N)^{\sigma}))}
               + \|y_i\|_{H^{\sigma}(\Omega)}
               + \|y_i\|_{L^\infty(Q_T)}
               \le C,
               \label{ineq: global estimates}
            \end{align}
            for $i=1,2$
        \label{Thm: globalexistence}
    \end{theorem}
    \begin{proof}
        Since $A$ is self-adjoint and dissipative on the Hilbert space $D(A)$, $y_0=(S_0,I_0)\in D(A)$ by hypothesis, and by Lemma \ref{lemma F Lipschitz} $F$ is Lipschitz continuous in $y$ uniformly with respect to $t$, then, by Lemma \ref{lemma : Local and Global Existence} part (ii), we conclude that the mild solution of problem \eqref{Eq: Abstract_FDE_control} is indeed the strong global solution. In this case $y=(y_1,y_2)\in W^{1,2}([0,T];D_A(\Omega))\cap L^2(0,T;D(A))$, i.e. $y_i\in L^2(0,T;D((-\Delta_N)^\sigma))$ for $i=1,2$.

        It remains to show that the estimate \eqref{ineq: global estimates} hold. Taking the scalar product of \eqref{Eq : FDE_control_system1} with itself and with $-d_1(-\Delta_N)^\sigma y_1$ in $L^2(Q_T)$ and combining the two identities, we obtain, for every $t\in[0,T]$, 
            \begin{align}
                \int_0^t\int_\Omega \left|\frac{\partial y_1}{\partial s}\right|^2 \ dx\ ds
                &\ +d_1^2 \int_0^t\int_\Omega |(-\Delta_N)^\sigma y_1|^2 \ dx\ ds
                \ +d_1 \left\|(-\Delta_N)^{\sigma/2} y_1(t) \right\|_{L^2(\Omega)}^2 
                \nonumber
                \\
                &\le d_1 \left\|(-\Delta_N)^{\sigma/2}  y_{10}\right\|_{L^2(\Omega)}^2+\left\|F_1\right\|_{L^2(Q_T)}^2.
                \label{eq: bound y1 with RHS}
            \end{align}
        We notice that since $y_{10}\in D\big((-\Delta_N)^\sigma\big)$ then $y_{10}\in H^{\sigma}(\Omega)$.
        %
        Observe also that since $y_i\in L^\infty(Q_T)$ then $F_1(y_1,y_2)\in L^\infty(Q_T)$ which implies $F_1(y_1,y_2)\in L^2(Q_T)$. Therefore, from \eqref{eq: bound y1 with RHS}, we have
            \begin{align}
                \int_0^t\int_\Omega \left|\frac{\partial y_1}{\partial s}\right|^2 \ dx\ ds
                &\ +d_1^2 \int_0^t\int_\Omega |(-\Delta_N)^\sigma y_1|^2 \ dx\ ds
                \ +d_1 \left\|(-\Delta_N)^{\sigma/2} y_1(t)\right\|_{L^2(\Omega)}^2 
                \le C < \infty,
                \label{eq: bound y1 no RHS}
            \end{align}
            with $C>0$ independent of $t$.
        Hence, we conclude that 
            \begin{align}
               \left\|\frac{\partial y_1}{\partial t}\right\|_{L^2(Q_T)} \le C,
               \quad
               \|y_1\|_{L^2(0,T;D((-\Delta_N)^\sigma))} \le C, 
               \quad
               \|y_1\|_{H^{\sigma}(\Omega)}\le C,
               \label{ineq: individual y_1 estimate}
            \end{align}
        which gives \eqref{ineq: global estimates}. Moreover, by taking the supremum of the last estimate of \eqref{ineq: individual y_1 estimate} over $[0,T]$, we obtain 
            \begin{align}
                y_1 \in L^\infty([0,T];H^{\sigma}(\Omega)).
                \label{eq : inclusion y1}
            \end{align}
        Combining the second estimate of last estimate of \eqref{ineq: individual y_1 estimate} and \eqref{eq : inclusion y1} we conclude that $y_1\in N(T,\Omega)$.
        The proof of the estimate \eqref{ineq: global estimates} for $y_2$, and that $y_2\in N(T,\Omega)$ can be obtained in similar manner as the proof of $y_1$.
    \end{proof}

\section{The existence of optimal solution}
\label{sec:con_existence_op}

In this section, we  establish the existence of an optimal solution for control problem \eqref{Eq: Objective_Functional}. More precisely, we  show that there exists a sequence $(y_n,u_n)=(y_{1}^n, y_{2}^n, u_{1}^n, u_{2}^n)$ where $y_{1}^n, y_{2}^n, u_{1}^n, u_{2}^n$ satisfy \eqref{eq : F1 F2 RHS}--\eqref{Eq: Abstract_FDE_control} with the limit $y_1^*,y_2^*,u_1^*,u_2^*$ that solve \eqref{eq : F1 F2 RHS}--\eqref{Eq: Abstract_FDE_control} and the convergence of $y_2,u_1,u_2$ is defined in the norm of $J(y,u)$ (i.e $y_{2}^n\to y_2^*$ in $L^2(Q_T)$, $u_{1}^n\to u_1^*$ in $L^2(Q_\omega)$ and $u_{2}^n\to u_2^*$ in $L^2(Q_T)$).

Let
    \begin{align*}
        \mathcal{Y} &= W^{1,2}([0,T], D_A(\Omega))\ \cap \ N(T,\Omega)^2\ \cap\ L^\infty(Q_T)^2;
        \\
        \mathcal{U} &=  L^2\left(0, T ; D_A(\Omega)\right);
        \\
        \mathcal{X} &=\mathcal{Y}\times \mathcal{U};
        \\
        \mathcal{P} &=  W^{1,2}([0,T], D_A(\Omega)) \times D_A(\Omega).
    \end{align*}    

    \begin{theorem}
        Under the hypotheses of Theorem \ref{Thm: globalexistence}, \eqref{Eq: Uadmissible system 2} admits an optimal solution $(y^*, u^*)\in\mathcal{X}$
        \label{thm: exist opt sol}
    \end{theorem}

    \begin{proof}
        Let $J^*=\inf_{(y,u)\in \mathcal{X}}\{J(y, u)\}:= {J(y_1^*,y_2^*, u_1^*, u_2^*)}$
        with $J$ is as defined in \eqref{Eq: Objective_Functional}. We will show that if there exists $u_n=(u_{1}^n, u_{2}^n)\in U_{ad}$, and $u_{1}^n\to u_1^*$ in $L^2(Q_\omega)$ and $u_{2}^n\to u_2^*$ in $L^2(Q_T)$ then $(u_1^*,u_2^*)$ is indeed the minimizer. We shall construct the sequence.

        Since $u_1,u_2, y_1, y_2$ are bounded uniformly in $L^\infty(Q_T)$ then $J^*=\inf_{(y, u)\in\mathcal{X}}\{J(y, u)\}$ is finite. Therefore, there exist a sequence $(y_n,u_n)$ with $y_n=(y_{1}^n,y_{2}^n)\in W^{1,2}([0,T];D_A(\Omega))$ solution of
            \begin{align}
                \frac{\partial y_{i}^n}{\partial t} = -d_i(-\Delta_N)^\sigma y_{i}^n+F_i(y_n), 
                \qquad i=1,2 
                \label{eq : FDE yn}
            \end{align}
    where $F_i$ for $i=1,2$ are given in \eqref{eq : F1 F2 RHS}. 

    Before proceeding further, it is noteworthy that, from \eqref{ineq: global estimates}, there exist $C>0$ such that
    \begin{align}  
       \left\|\frac{\partial y_{i}^n}{\partial t}\right\|_{L^2(Q_T)}\le C,
       \quad
       \|y_{i}^n\|_{L^2(0,T;D((-\Delta_N)^{\sigma}))}\le C, 
       \quad
       \|y_{i}^n\|_{H^{\sigma}(\Omega)} \le C,
       \quad
       \|y_{i}^n\|_{L^\infty(Q_T)}\le C,
       \label{ineq: global estimates y_n}
    \end{align}
    We outline the proof in the following steps: 

    \textbf{Step 1.} First, to show that $y_{2}^n$ converges to a limit in the first term of functional $J$, we prove that $y_{2}^n\to y_2^*$ in $L^2(Q_T)$. We note that since $\{y_{i}^n\}_{n\in\mathbb{N}}\subset L^{\infty}(Q_T)$ and  $\{y_{i}^n\}_{n\in\mathbb{N}}\subset H^{\sigma}(\Omega)$ then $\{y_{i}^n\}_{n\in\mathbb{N}}\subset L^{2}(Q_T)$. Thus, for $i=1,2$, the sequence $\{\|y_{i}^n\|_{L^2(\Omega)}(t)\}_{n\in\mathbb{N}}$ is uniformly bounded in $L^{2}(Q_T)$. Now it remains to show that $\{\|y_{i}^n\|_{L^2(\Omega)}(t)\}_{n\in\mathbb{N}}$ is equicontinuous.

    For any $\epsilon>0$, we want to choose $\delta>0$ such that for any $t,\tilde{t}$ with $|t-\tilde{t}|\le \delta$,
        \begin{align}
            \left| 
            \|y_{i}^n\|^2_{L^2(\Omega)}(t) - \|y_{i}^n\|^2_{L^2(\Omega)}(\tilde{t}) 
            \right|
            =
            \left| 
            \int_\Omega |y_{i}^n(x,t)|^2 dx- \int_\Omega |y_{i}^n(x,\tilde{t})|^2 dx
            \right|
            <\epsilon
        \end{align}
    for $i=1,\ 2$.
    
    \textbf{Step 2.} We show that, for $i=1,\ 2$, the sequence $\{\|y_{i}^n\|^2_{L^{2}(\Omega)}(t)\}_{n\in\mathbb{N}}\subset C([0,T])$ is relatively compact by Arzela-Ascoli theorem. 

    \textbf{Step 3.} We conclude that since $\{\|y_{i}^n\|^2_{L^{2}(\Omega)}(t)\}_{n\in\mathbb{N}}\subset C([0,T])$, then there exists a subsequence $\{y_i^{n_k}\}_{k\in\mathbb{N}}$ of $\{y_i^n\}_{n\in\mathbb{N}}$ such that $\|y_i^{n_k}\|^2_{L^{2}(\Omega)}(t)$ converges to $\|y_i^*(t)\|^2_{L^{2}(\Omega)}$ uniformly in $t\in[0,T]$, for $i=1,\ 2$. Moreover, we need to also show that \eqref{ineq: global estimates y_n} is also satisfied by $y^*=(y_1^*,y_2^*)$.


    Multiplying the $i$-th equation of \eqref{eq : FDE yn} by $y_i^n$ and integrating over $\Omega\times [0,t]$ and $\Omega\times [0,\tilde t]$ and for $t<\tilde t$, a standard estimate using \eqref{ineq: global estimates y_n} yields
        
        \begin{align*}
            \int_\Omega |y_{i}^n(x,t)|^2-|y_{i}^n(x,\tilde{t})|^2 dx
            &=
            -2d_i \int_{\tilde{t}}^t \left\|(-\Delta_N)^{\sigma/2} y_{i}^n\right\|^2_{L^2(\Omega)}(s)\ ds
            \\
            &\quad
            + 2\int_{\tilde{t}}^t \int_{\Omega} y_{i}^n(x,s)  F_i(y_n,s)\ dxds,
            \intertext{so,}
            \left|
            \left\|y_{i}^n\right\|^2_{L^2(\Omega)}(t) - \left\|y_{i}^n\right\|^2_{L^2(\Omega)}(\tilde{t})  
            \right|
            &\le 2 d_i 
            \left\|(-\Delta_N)^{\sigma/2} y_{i}^n\right\|^2_{L^2(\Omega)}\ |t-\tilde{t}| 
            \\
            &\quad
            +2\|y_{i}^nF_i(y^n)\|_{L^\infty(Q_T)}\ |\Omega|\ |t-\tilde{t}|
            \\
            &\le  C_i |t-\tilde{t}|<\varepsilon, \qquad\text{for some $C_i>0$},
        \end{align*}
        by choosing $\delta=\epsilon/C_i$. Hence we conclude that the sequence $\{\|y_{i}^n\|_{L^{2}(\Omega)}(t)\}_{n\in\mathbb{N}}$ is relatively compact in $C([0,T])$ or $\{y_{i}^n\}_{n\in\mathbb{N}}$ is relatively compact in $C([0,T], L^2(\Omega))$.

         Consequently,
            \begin{itemize}
                \item There exists a subsequence $\{\|y_i^{n_k}\|\}_{n_k=1}^\infty$ of $\{\|y_{i}^n\|_{L^{2}(\Omega)}\}_{n=1}^\infty$ such that 
                    \begin{align}
                        \|y_i^{n_k}\|_{L^{2}(\Omega)} \rightarrow \|y_i^*\|_{L^{2}(\Omega)} 
                        \qquad \text{in} \quad
                        L^2(\Omega)
                        \qquad \text{uniformly with respect to }t.
                    \end{align}
    
                \item 
                By Lebesgue's theorem, we have 
                    \begin{align}
                        \lim_{n_k\to\infty} \int_0^T \int_{\Omega} |y_i^{n_k}|^2 \, dx\, dt
                        = \int_0^T \lim_{n_k\to\infty} \int_{\Omega}  |y_i^{n_k}|^2 \, dx\, dt
                        = \int_0^T \int_{\Omega} |y_{i}^*|^2 \, dx\, dt,
                        \label{eq : ynk conv in L2(Q_T)}
                    \end{align}
                or $\|y_i^{n_k}\|_{L^2(Q_T)} \xrightarrow{n_k\to\infty} \|y_i^*\|_{L^2(Q_T)}$ for $i=1,2$.
            \end{itemize}
        Since $y_i^{n_k}\in {L^\infty(Q_T)}$, then $y_i^{n_k}$ is bounded in $L^2(Q_T)$. Thus, 
            \begin{align}
                y_i^{n_k} \rightharpoonup y_i^*
                \qquad \text{in} \quad L^2(Q_T),
                \qquad i=1,2.
                \label{eq : ynk weak conv in L2(Q_T)}
            \end{align}
        Combining the norm convergence \eqref{eq : ynk conv in L2(Q_T)} and the weak convergence \eqref{eq : ynk weak conv in L2(Q_T)}, we have the strong convergence, namely
            \begin{align}
                y_{i}^{n_k} \rightarrow y_{i}^*
                \qquad \text{in} \quad L^2(Q_T),
                \qquad i=1,2.
            \end{align}
        However, notice that up until this point we have not shown that $y_i^*$ is a solution of \eqref{Eq: Abstract_FDE_control}.
           For simplicity of notation, we rename the subsequence $\{y_i^{n_k}\}_{n_k=1}^\infty$ satisfying \eqref{eq : ynk conv in L2(Q_T)} by $\{y_{i}^n\}_{n=1}^\infty$ for $i=1,2$.
        Now we show that $y^*$ is also a solution of \eqref{Eq: Abstract_FDE_control} that satisfies \eqref{ineq: global estimates}, i.e. 
            \begin{align*}
                y^*\in W^{1,2}([0,T], D_A(\Omega))
                \intertext{and}
                y^*\in \left(L^2\big(0,T;D((-\Delta_N)^\sigma)\big) \right)^2 
                &\cap \left(L^{\infty}\big([0,T]; H^{\sigma}(\Omega) \big)\right)^2.
            \end{align*}

        \begin{enumerate}
            \item Since  $\left\|\frac{\partial y_{i}^n}{\partial t}\right\|_{L^2(Q_T)}\le C$, then $\frac{\partial y_i^{n_k}}{\partial t} \rightharpoonup \frac{\partial y_i^*}{\partial t}$ in $L^2(Q_T)$.
            \item 
            For any $\varphi\in C_0^\infty(\Omega)$, using the integration by parts formula \eqref{eq: int by parts lap}, we have
                \begin{align}
                     \lim_{n\to\infty} \int_0^T \int_{\Omega} \varphi\ (-\Delta_N)^\sigma y_{i}^n \ dx\ dt
                     &= \lim_{n\to\infty}\int_0^T \int_{\Omega}  (-\Delta_N)^\sigma \varphi\ y_{i}^n \ dx\ dt
                     \nonumber \\
                     &= \int_0^T \int_{\Omega}  (-\Delta_N)^\sigma \varphi\ y_{i}^* \ dx\ dt
                     \nonumber \\
                     &= \int_0^T \int_{\Omega}  \varphi\ (-\Delta_N)^\sigma y_{i}^* \ dx\ dt,
                     \label{eq: weak conv ynt}
                \end{align}
            where the second equality follows from the strong convergence $y_i^n\to y_i^*$ in $L^2(Q_T)$ established in \eqref{eq : ynk conv in L2(Q_T)}.
            
            \item Next, the convergence of every product term $y_{i}^ny_{j}^n$ on the right-hand side of \eqref{Eq: Abstract_FDE_control} for $i,j=1,2$, is immediate, namely
                        $$y_{i}^ny_{j}^n\to y_{i}^*y_{j}^* \qquad \text{in } L^2(Q_T).$$ 
                \item We show that $ u_{1}^n y_{1}^n$ to $u_{1}^* y_{1}^*$ in $L^2(Q_{\omega})$ and  $u_{2}^n y_{2}^n$ to $u_{2}^* y_{2}^*$ in $L^2(Q_T)$.
                \begin{itemize}
                    \item
                    Observe that since $u_{1}^n\in L^\infty(Q_{\omega})$ and $u_{2}^n\in L^\infty(Q_T)$, then $u_{1}^n$ is bounded in $L^{2}(Q_\omega)$ and $u_{2}^n$ is bounded in $L^{2}(Q_T)$. Therefore, using Lebesgue's theorem similar as before, we conclude that $u_{1}^n \rightharpoonup u_1^*$ in $L^2(Q_{\omega})$ and $u_{2}^n \rightharpoonup u_2^*$ in $L^2(Q_T)$.
                    Since $U_{ad}$ is a closed and convex set, it contains all of its limit points. In particular, it contains its weakly limit points. Thus, $(u_1^*,u_2^*)\in U_{ad}$.

                    \item Since $u_{1}^n \rightharpoonup u_1^*$ in $L^2(Q_{\omega})$, then for all $\varphi\in L^2(Q_\omega)$ we have
                        \begin{align}
                            \int_0^T\int_\Omega \chi_\omega u_{1}^n\varphi\ dx dt \to \int_0^T\int_\Omega \chi_\omega u_{1}^*\varphi\ dx dt.
                            \label{eq: u1n phi to  u1 phi}
                        \end{align}
                   Observe that
                        \begin{align*}
                            \|u_{1}^ny_{1}^n - u_{1}^*y_{1}^*\|^2_{L^2(Q_\omega)} 
                            &\le 2\left(\|u_{1}^n\|^2_{L^2(Q_{\omega})}\|y_{1}^n-y_1\|^2_{L^2(Q_{\omega})} 
                            +\|(u_{1}^n-u_1)y_1\|^2_{L^2(Q_{\omega})}
                            \right)
                            \intertext{Notice that, the first term goes to $0$ as $n\to \infty$ due to the strong convergence of $y_{1}^n\to y_!^*$ in $L^2(Q_T)$. Now, since $u_{1}^n,y_{1}^n\in L^\infty (Q_T)$, then}
                            \|(u_{1}^n-u_1)y_1\|^2_{L^2(Q_{\omega})}
                            &= \int_0^T\int_\omega |(u_{1}^n-u_1)y_1|^2 dx dt\\
                            &= \|y_{1}\|_{\infty}\left(\|u_{1}^n\|_{\infty)}+\|u_1\|_{\infty}\right)
                            \int_0^T\int_\omega |(u_{1}^n-u_1)| |y_1| dx dt\\
                            &= C \int_0^T\int_\Omega \chi_\omega u_{1}^ny_1 -\chi_\omega u_1 y_1 dx dt
                            \to 0,
                            \quad\text{as }n\to\infty,
                        \end{align*}
                        where we have used \eqref{eq: u1n phi to  u1 phi}. Hence, we assert that 
                            \begin{align}
                                \chi_\omega u_{1}^ny_{1}^n\xrightarrow{n\to \infty} \chi_\omega u_1^*y_1^*, \qquad\text{in }\quad L^2(Q_\omega).
                            \end{align}
                        Similarly we conclude that
                            \begin{align}
                                u_{2}^ny_{2}^n \xrightarrow{n\to \infty} u_2^*y_2^*  \qquad\text{in }\quad L^2(Q_\omega).
                            \end{align}
                \end{itemize}
            
            Combining all the results on the right-hand side, we see that we have established a sequence $\{y_{n},u^n\}$ that satisfy \eqref{Eq: Abstract_FDE_control}. Passing \eqref{eq : FDE yn} to the limit as $n\to \infty$ in $L^2(Q_T)$, we see that $(y^*,u^*)$  solves the problem \eqref{Eq : FDE_control_system1}-\eqref{Eq : FDE_control_system2}. 
                \begin{align}
                    \frac{\partial y_{i}^n}{\partial t} 
                    &\rightharpoonup \frac{\partial y_i^*}{\partial t}, 
                    \qquad \text{in } L^2(Q_T),
                    \\
                    y_{i}^n &\rightharpoonup y_i^*
                    \qquad\quad \text{ in } L^2(0,T; D((-\Delta_N)^\sigma)),
                     \\
                    y_{i}^n &\rightharpoonup y_i^*
                    \qquad \quad\text{ in } H^{\sigma}(\Omega),
                    \\
                    y_{i}^n &\rightharpoonup y_i^*
                    \qquad\quad \text{ in } L^{\infty}(Q_T).
                \end{align}
            The last two conditions allow us to assert that $y_i^* \in L^\infty(0,T; H^{\sigma}(\Omega))$, while the first condition gives $\partial_t y_i^*\in L^2(Q_T)$, which leads us to conclude that $y^*\in W^{1,2}([0,T],D_A(\Omega))$.
            \end{enumerate}
           
          Lastly, we show that $(y^*,u^*)\in \mathcal{X}$ is a minimizer. Indeed,
            \begin{align*}
                J(y^*,u^*)&= \frac{\rho}{2}\|y_2^*\|^2_{L^2(Q_T)} 
                + \frac{\sigma_1}{2}\|u_1^*\|^2_{L^2(Q_\omega)} 
                + \frac{\sigma_2}{2}\|u_2^*\|^2_{L^2(Q_T)}
                \\
                &\le \liminf_{n\to \infty } \left(
                \frac{\rho}{2}\|y_{2}^n\|^2_{L^2(Q_T)} 
                + \frac{\sigma_1}{2}\|u_{1}^{n}\|^2_{L^2(Q_\omega)} 
                + \frac{\sigma_2}{2}\|u_{2}^{n}\|^2_{L^2(Q_T)}\right)
                \\
                &= \lim_{n\to \infty } 
                \frac{\rho}{2}\|y_{2}^n\|^2_{L^2(Q_T)} 
                + \frac{\sigma_1}{2}\|u_{1}^{n}\|^2_{L^2(Q_\omega)} 
                + \frac{\sigma_2}{2}\|u_{2}^{n}\|^2_{L^2(Q_T)}
                \\
                &= \inf_{(y,u)\in\mathcal{X}} J(y,u).
            \end{align*}
        This shows that $J$ attains its minimum at $(y^*,u^*)$.
    \end{proof}

\section{Necessary Optimality Condition}
\label{sec:con_necessary_op}

    In this section, we specify the optimality condition of the problem \eqref{Eq : FDE_control_system1}-\eqref{Eq: Uadmissible system 2} and we give a characterization of the optimal control. To establish the characterization, we need to find the dual problem and the directional derivative of $y$ with respect to $ u^*$.
    For this purpose, let
    \begin{align*}
        u=
            \begin{pmatrix}
                u_1 \\
                u_2
            \end{pmatrix}
        \in U_{ad} 
        \quad \text { and } \quad
         u^*=
            \begin{pmatrix}
                u_1^* \\
                u_2^*
            \end{pmatrix}
        \in U_{ad}. 
    \end{align*}
    To compute the Gateaux derivative of $y$, seen as a function of $u$
        \begin{align*}
            y: U_{ad} & \longrightarrow W^{1,2}([0, T], D_A(\Omega)) \\
            u \  & \longmapsto \  y(u),
        \end{align*}
            Let $ y^*,  u^*$ be an optimal pair of problems \eqref{Eq : FDE_control_system1}-\eqref{Eq: Uadmissible system 2}, and let
                \begin{align}
                    u^{\varepsilon}=u^{*} + \varepsilon u \in U_{ad},
                    \qquad \varepsilon>0,
                    \label{eq: u_epsilon}
                \end{align}
            be a control function with $u \in L^2 \left(0, T ; L^2(\Omega)\right)$ and $u \in U_{ad}$. 
            %
            Denote by $y^\varepsilon=(y_1^\varepsilon,y_2^\varepsilon)=(y_1,y_2)(u^\varepsilon)$ and $y^*=(y_1^*,y_2^*)=(y_1,y_2)(u^*)$ the corresponding solutions of \eqref{Eq : FDE_control_system1}-\eqref{Eq: IC_system}

    Let
        \begin{align*}
            P(y)=
            \begin{pmatrix}
                y_1 y_2 \psi(y_2)\\
                -y_1 y_2 \psi(y_2)
            \end{pmatrix}
        \qquad\text{and}\qquad
            L= 
            \begin{pmatrix}
                - \chi_\omega y_1^*  &0\\
                    0                &- y_2^*  
            \end{pmatrix}.
        \end{align*}
        
        \begin{theorem}
            The mapping $y: U_{ad} \rightarrow W^{1,2}([0, T], D_A(\Omega))$ with $y_i\in N(T, \Omega)$ for $i=1,2$, is Gateaux differentiable with respect to $ u^*$. For any direction $u\in U_{ad}$, $y'( u^*)u=z$ is the unique solution in $W^{1,2}([0, T], D_A(\Omega))$ with $z_i \in N(T, \Omega)$ of the following equation
            \begin{align*}
                \frac{\partial z}{\partial t}  & =A z+B z+L u, \\
                z(0) & =0,
            \end{align*}
            where $B =\frac{\partial P}{\partial y}(y^*)$.
        \end{theorem}

        \begin{proof}
           Let
            \begin{align*}
                y^\varepsilon= y^* + \varepsilon z^\varepsilon,
            \end{align*}
            then
                \begin{align*}
                    z^\varepsilon=\frac{y^\varepsilon- y^*}{\varepsilon}.
                \end{align*} 
            Subtracting the system  \eqref{Eq : FDE_control_system1}-\eqref{Eq: IC_system} corresponding to $ y^*$ from the system corresponding to $y^\varepsilon$ and dividing by $\varepsilon$, yields
                \begin{align}
                    \frac{\partial z_1^\varepsilon}{\partial t} &= -d_1(-\Delta_N)^\sigma z_1^\varepsilon 
                    -\frac{
                    P(y_1^\varepsilon, y_2^\varepsilon)-P(y_1^*, y_2^*)
                    }{\varepsilon} 
                    -a_1 z_1^\varepsilon 
                    - \chi_\omega (u_1^\varepsilon z_1^\varepsilon+u_1 y_1^*), 
                    \label{eq: dz1^epsilon 2}
                    \\
                    \frac{\partial z_2^\varepsilon}{\partial t} &= -d_2(-\Delta_N)^\sigma z_2^\varepsilon 
                    +\frac{
                    P(y_1^\varepsilon, y_2^\varepsilon)-P(y_1^*, y_2^*)
                    }{\varepsilon} 
                    -a_2 z_2^\varepsilon - (u_2^\varepsilon z_2^\varepsilon + u_2 y_2^*).
                    \label{eq: dz2^epsilon 2}
                \end{align}
            By letting 
                \begin{align}
                    M_1^\varepsilon = \frac{P(y_1^\varepsilon, y_2^\varepsilon)-P(y_1^*, y_2^\varepsilon)}{y_1^\varepsilon-y_1^{*}},
                    \quad\text{and}\quad
                    M_2^\varepsilon = \frac{P(y_1^*, y_2^\varepsilon)-P(y_1^*, y_2^*)}{y_2^\varepsilon-y_2^{*}} ,
                \end{align}
            then 
                \begin{align*}
                    \frac{P(y_1^\varepsilon, y_2^{\varepsilon})-P(y_1^*, y_2^*)}{\varepsilon}
                    = M_1^\varepsilon z_1^\varepsilon + M_2^\varepsilon z_2^\varepsilon.
                \end{align*}
            Thus, \eqref{eq: dz1^epsilon 2}-\eqref{eq: dz2^epsilon 2} becomes
                \begin{align}
                    \frac{\partial z_1^\varepsilon}{\partial t} &= -d_1(-\Delta_N)^\sigma z_1^\varepsilon 
                    - M_1^\varepsilon z_1^\varepsilon 
                    - M_2^\varepsilon z_2^\varepsilon
                    -a_1 z_1^\varepsilon 
                    - \chi_\omega u_1^\varepsilon z_1^\varepsilon
                    -\chi_\omega u_1 y_1^*, 
                    \label{eq: dz1^epsilon 3}
                    \\
                    \frac{\partial z_2^\varepsilon}{\partial t} &= -d_2(-\Delta_N)^\sigma z_2^\varepsilon 
                    + M_1^\varepsilon z_1^\varepsilon 
                    + M_2^\varepsilon z_2^\varepsilon
                    -a_2 z_2^\varepsilon 
                    - u_2^\varepsilon z_2^\varepsilon - u_2 y_2^*;
                    \label{eq: dz2^epsilon 3}
                \end{align}
            or  
                \begin{align}
                    \begin{cases}
                        \frac{\partial z^\varepsilon }{\partial t}
                     =Az^\varepsilon + B^\varepsilon z^\varepsilon + Lu,\\
                     z^\varepsilon(0)=0, 
                    \end{cases}
                    \label{Eq: Abstract_FDE_control z^epsilon}
                \end{align}
            with $z^\varepsilon=(z_1^\varepsilon,z_2^\varepsilon)$, where
                \begin{align}
                     B^\varepsilon = 
                         \begin{pmatrix}
                            -M_1^\varepsilon -a_1 -\chi_\omega u_1^\varepsilon     &-M_2^\varepsilon\\
                            M_1^\varepsilon     &M_2^\varepsilon -a_2 - u_2^\varepsilon 
                         \end{pmatrix}
                    \qquad \text{and}\qquad
                    L=
                        \begin{pmatrix}
                            - \chi_\omega y_1^*  &0\\
                            0                     &- y_2^*  
                         \end{pmatrix}.
                \end{align}
            If $(\mathcal{S}(t))_{t\ge0}$ is the semigroup of contraction generated by $A$, then the solution of system \eqref{Eq: Abstract_FDE_control z^epsilon} can be written as
                \begin{align*}
                    z^{\varepsilon}(t)
                    =\int_0^t \mathcal{S}(t-s)\ B^{\varepsilon}(s) \ z^\epsilon(s)\ ds
                    +\int_0^t \mathcal{S}(t-s)\ L u(s)\ ds.
                \end{align*}
            We now show that $y_1^\varepsilon \rightarrow y_1^*$ as $\varepsilon\to 0$ in $L^2(Q_T)$.
        
                Notice that $B^{\varepsilon}$ is bounded uniformly with respect to $\varepsilon$, i.e.\
            $\left\|B^{\varepsilon}\right\|_{L^\infty(\Omega)}<C$ for some $C>0$; since $\mathcal{S}(t)$ is a
            contraction semigroup, $\|\mathcal{S}(t-s)\|\le 1$, and thus
                \begin{align*}
                    \|z_i^\varepsilon(t)\|_{L^{\infty}(\Omega)}
                    &\le \int_0^t \|\mathcal{S}(t-s)\|\ \|B^{\varepsilon}(s)\| \ \|z_i^\varepsilon(s)\|_{L^\infty(\Omega)} \ ds
                        +\int_0^t \|\mathcal{S}(t-s)\|\ \|L\|_{L^{\infty}(\Omega)}\ \|u(s)\|\ ds\\
                    &\le C\int_0^t \|z_i^\varepsilon(s)\|_{L^\infty(\Omega)}\,ds
                        + C'\|u\|_{L^\infty(Q_T)}\,t,
                \end{align*}
            where $C'>0$ depends only on $\|L\|_{L^\infty(\Omega)}$. By Gr\"onwall's inequality, this yields
                \begin{align*}
                    \|z_i^\varepsilon(t)\|_{L^\infty(\Omega)} \le C'\|u\|_{L^\infty(Q_T)}\,t\,e^{Ct}
                    \le C'\|u\|_{L^\infty(Q_T)}\,T\,e^{CT} =: L_i,
                \end{align*}
            for some $L_i>0$, independent of $\varepsilon$, with $i=1,2$.
            Hence
                \begin{align*}
                \left\|z_i^{\varepsilon}\right\|_{L^{2}(Q_T)}^2
                    =\int_0^T \int_\Omega \left|z_i^\varepsilon \right|^2 dx\ ds 
                    \le L_i^2 \ T \ |\Omega|
                    =K_i^2.
                \end{align*}
            Taking $K=\max\left(K_1, K_2\right),$ we have 
                \begin{align*}
                    \left\|z_i ^{\varepsilon}\right\|_{L^{2}(Q_T)}<K<\infty .
                \end{align*}
            Therefore, we may write
                \begin{align*}
                    \|y_i^{\varepsilon}-y_i^*\|_{L^{2}(Q_T)}
                    =\varepsilon\ \|z_i^\varepsilon\|_{L^{2}(Q_T)}<\varepsilon K,    
                    \qquad i=1,2.
                \end{align*}
            Taking the limit as $\varepsilon \rightarrow 0$, we obtain
                \begin{align*}
                    y_i^{\varepsilon} \longrightarrow y_i^* 
                    \quad\text { in } \quad L^2(Q_T).
                \end{align*}
               In this case, we obtain (as $\varepsilon \to 0$, i.e., $y_i^\varepsilon \rightarrow y_i^*$)
                \begin{align*}
                    M_1^\varepsilon \longrightarrow M_1^*=\frac{\partial P}{\partial y_1}(y_1^*, y_2^*) 
                    \quad\text { and }\quad
                    M_2^{\varepsilon} \longrightarrow M_2^*=\frac{\partial P}{\partial y_2}\left(y_1^*, y_2^*\right) \text { in } L^2(Q_T).
                \end{align*}
            Thus, $B^{\varepsilon} \longrightarrow B^{*}$ element-wise in $L^2(Q_T)$ with 
                \begin{align}
                    B^*=
                        \begin{pmatrix}
                            -M_1^*-a_1-\chi_\omega u_1^* & -M_2^* \\
                            M_1^* & M_2^*-a_2-u_2^*
                        \end{pmatrix}.
                    \label{eq: matrix Bopt}
                \end{align}
            This means, as $\varepsilon \rightarrow 0$, the system \eqref{Eq: Abstract_FDE_control z^epsilon} goes to
            \begin{align*}
                \begin{cases}
                    \frac{\partial z^*}{\partial t}=A z^*+B^* z^*+L u,
                    \\
                    z^*(0)=0.
                \end{cases}
            \end{align*}
            Or, after renaming $z^* = z$ and $B^*=B$,
            \begin{align*}
                \begin{aligned}
                    & \frac{\partial z}{\partial t}=A z+B z+L u; \\
                    & z(0)=0.
                \end{aligned}
            \end{align*}
        \end{proof}

        Next, we give a characterization of the optimal control $u$.
        Let $ u^*\in \mathcal{U}$ be the optimal control associated with the state variable $ y^*\in \mathcal{Y}$. Let $p^*\in\mathcal{P}$ be the corresponding adjoint variable. First, we shall consider the Lagrangian of our problem (in the matrix form). Recall that 
            \begin{align*}
                \|y\|_{D_A(\Omega)}^2 := \|y_1\|_{L^2(\Omega)}^2 + \|y_2\|_{L^2(\Omega)}^2.
            \end{align*}
        
            Define
            $
                D=
                \begin{bmatrix}
                    0   &0\\
                    0   &1
                \end{bmatrix};
            $
            so that $\|y_2\|_{L^2(Q_T)}^2 = \int_0^T \langle Dy, Dy \rangle_{D_A(\Omega)}\, dt$
        
            Hence,
                \begin{align*}
                    J(y, u)&=\frac{\rho}{2} \int_0^T \left\langle D y, D y\right\rangle_{D_A(\Omega)} d t
                    +\int_0^T \langle  \tilde{u},  \tilde{u}\rangle_{D_A(\Omega)} d t,  
                    \intertext{ where } 
                    \tilde{u}&=
                    \begin{bmatrix}
                        \sqrt{\sigma_1/2}\  \chi_\omega(x)\ u_1 \\
                        \sqrt{\sigma_{2}/2} \ u_2
                    \end{bmatrix}.
                \end{align*}
        Then, the Lagrangian is given by
            \begin{align*}
                \mathcal{L} (y, u, p)
                &= J(y,u) + \left\langle e(y),p \right\rangle_{D_A(\Omega)}
                \\
                &= \frac{\rho}{2} \int_0^T \|Dy\|_{D_A(\Omega)}^2 dt  
                + \int_0^T \|\tilde{u}\|_{D_A(\Omega)}^2 dt
                \\
                &\quad
                +\int_0^T \int_\Omega p_{11}(x,t) 
                \left(
                \frac{\partial y_1}{\partial t} + d_1(-\Delta_N)^\sigma y_1 -F_1(y_1,y_2,u_1,u_2)
                \right)\ dx dt
                \\
                &\quad +\int_0^T \int_\Omega p_{12}(x,t) \left(
                \frac{\partial y_2}{\partial t} + d_2(-\Delta_N)^\sigma y_2 -F_2(y_1,y_2,u_1,u_2)
                \right)\ dx dt
                \\
                &\quad 
                + \int_\Omega y_1(0) p_{21}(x) dx + \int_\Omega y_2(0) p_{22}(x) dx.
            \end{align*}
        Let $(y^*,  u^*, p^*)\in \mathcal{Y}\times\mathcal{U}\times\mathcal{P}$ be the optimal value; then it must satisfy 
            \begin{align}
                \mathcal{L}_p (y^*,  u^*, p^*)=0 
                \text{ in } \mathcal{P}',\quad
                \mathcal{L}_y (y^*,  u^*, p^*)=0 
                \text{ in } \mathcal{Y}',\quad
                \mathcal{L}_u (y^*,  u^*, p^*)=0 
                \text{ in } \mathcal{U}',
            \end{align}
        where $\mathcal{P}',\mathcal{Y}',$ and $ \mathcal{U}'$ are the dual spaces of $\mathcal{P},\mathcal{Y},$ and $\mathcal{U}$, respectively.

Consider the adjoint/dual equation $\mathcal{L}_y (y^*,  u^*, p^*)=0 $ in $\mathcal{Y}'$, that is, 
    \begin{align*}
        \langle \mathcal{L}_y (y^*,  u^*, p^*) ,y\rangle = 0, 
        \qquad\text{for all } y\in\mathcal{Y}.
    \end{align*}
Thus, since $\left. \frac{\partial F}{\partial y}\right|_{(y^*,u^*,p^*)}= B$, then 
    \begin{align*}
       0 &=\langle \mathcal{L}_y (y^*,  u^*, p^*) ,y\rangle_{\mathcal{Y}',\mathcal{Y}} 
        &= \rho \int_0^T \left\langle D^* Dy^*,y \right\rangle dt
       + \int_0^T \left\langle -\frac{\partial p}{\partial t}- Ap -B^*p ,y \right\rangle dt
    \end{align*}
provided that $p(T)=0$. Hence, the dual/adjoint equation is given by
    \begin{align}
        \begin{cases}
            -\frac{\partial p}{\partial t}-Ap-B^* p = \rho\ D^* D y^*, \\
            p(T, x)=0,
         \end{cases}
         \label{Eq: Abstract_FDE_control_dual/adjoint}
    \end{align}
where $D^*$ and $B^*$ denote the adjoint of $D$ and $B$, respectively. 

Then, we have the following lemma.
    \begin{lemma}
        Under the hypotheses of Theorem \ref{thm: exist opt sol}, if $(y^*, u^*)$ is an optimal pair, then there exists a unique strong solution $p \in W^{1,2}(0, T ; D_A(\Omega))$ of the system \eqref{Eq: Abstract_FDE_control_dual/adjoint} such that $p_i \in N(T, \Omega)$, for $i=1,2$.
    \end{lemma}
    \begin{proof}
        Let 
        $$s=T-t
        \quad \Rightarrow \quad
        p_i(t=T, x)=q_i(s=0, x).
        $$
        More precisely, let $q_i(s, x)=p_i(T-t, x)$. Then we have
            \begin{align*}
                q_i(0, x)=p_i(T, x)=0 \quad \text { for } \quad i= 1,2,
            \end{align*}
        with this change of variables, we have
            \begin{align*}
                \frac{\partial p}{\partial t}
                =\frac{\partial q}{\partial s} \cdot \frac{\partial s}{\partial t}
                =\frac{\partial q}{\partial s} \cdot \frac{\partial (T-t)}{\partial t}
                =-\frac{\partial q}{\partial s};
            \end{align*}
        hence, \eqref{Eq: Abstract_FDE_control_dual} becomes
           \begin{align}
                \begin{cases}
                    \frac{\partial q}{\partial s}-Aq-B^* q = \rho\ D^* D y^* ,\\
                    q(0, x)=0.
                 \end{cases}
             \label{Eq: Abstract_FDE_control_dual}
            \end{align}
        We see that this system is similar to \eqref{Eq: Abstract_FDE_control}. By repeating the argument of Theorem \eqref{Thm: globalexistence}, we prove the existence.
    \end{proof}

    Next, we give a characterization of the optimal control $ u^*_1,  u^*_2$.

    \begin{theorem}
        Let $ u^*$ be an optimal control of \eqref{Eq : FDE_control_system1}-\eqref{Eq: Objective_Functional} and let $ y^* \in W^{1, 2}\left([0, T] ; D_A(\Omega)\right)$ with $ y^*_i \in N(T, \Omega)$ for $i=1,2$ be the optimal state, that is, $ y^*$ is the solution of \eqref{eq : F1 F2 RHS}-\eqref{Eq: Abstract_FDE_control} with control $ u^*$, then,
        \begin{align}
            \begin{cases}
            u^*_1 &=\min \left(U_1, \max \left(0, \frac{\chi_\omega y_1^* p_1}{\sigma_1}\right)\right),
            \\
             u^*_2 &=\min \left(U_2, \max \left(0,  \frac{ y_2^* p_2}{\sigma_2}\right)\right).
            \end{cases}
        \label{eq: u* (optimal)}
        \end{align}
    \end{theorem}

    \begin{proof}
        Let $ u^*\in\mathcal{U}$ be the optimal control and $ y^*=\left(y_1^*, y_2^*\right)=\left(y_1, y_2\right)( u^*)=y\left( u^*\right)\in\mathcal{Y}$ be the corresponding solution of \eqref{eq : F1 F2 RHS}-\eqref{Eq: Abstract_FDE_control}. We first show that $J$ is Gateaux differentiable at $ u^*\in\mathcal{U}$, that is, 
            \begin{align*}
            J^{\prime}\left( u^*\right) h
            =\lim _{\varepsilon \rightarrow 0} \frac{1}{\varepsilon}\left( J\left( u^* +\varepsilon h\right)-J\left( u^*\right) \right).
        \end{align*}
        
        Let $y^{\varepsilon}=\left(y_1^{\varepsilon}, y_2^{\varepsilon}\right)=\left(y_1, y_2\right)\left(u^{\varepsilon}\right)=y\left(u^{\varepsilon}\right)$ be the solutions of $(1.1)-(1.3)$ that correspond to 
        $u^{\varepsilon} =  u^* + \varepsilon h \in U_{ad}$ and $h=(h_1,h_2) \in\left(L^2(Q_T)\right)^2$. 
        Notice that, $u^\varepsilon- u^*=\varepsilon h$. Then, 
            \begin{align*}
                J^{\prime}\left( u^*\right) h
                &=\lim _{\varepsilon \rightarrow 0} \frac{1}{\varepsilon}
                \left(
                \frac{\rho}{2}\left(\|y_2^\varepsilon\|^2_{L^2(Q_T)}- \|y_2^*\|^2_{L^2(Q_T)}\right)
                    + \frac{\sigma_1}{2} \left(\|u_1^\varepsilon\|^2_{L^2(Q_\omega)}-\|u_1^*\|^2_{L^2(Q_\omega)}\right)
                    \right.
                    \\
                    &\qquad\qquad\left.
                    + \frac{\sigma_2}{2} (\|u_2^\varepsilon\|^2_{L^2(Q_T)}-\|u_2^*\|^2_{L^2(Q_T)})
                    \right).
            \end{align*}
        Observe that 
            \begin{align*}
                \lim _{\varepsilon \rightarrow 0} \frac{\|y_2^\varepsilon\|_{L^2(\Omega)}^2- \|y_2^*\|_{L^2(\Omega)}^2}{\varepsilon}
                = \langle 2 y_2'( u^*),y_2^*  \rangle_{L^2(\Omega)},
                \qquad 
                \left. \frac{\partial}{\partial u_2}\|u_2\|_{L^2(\Omega)}^2\right|_{u=u_2^*} 
                = 2 \langle u_2^*,h\rangle_{L^2(\Omega)}.
            \end{align*}
        and 
            \begin{align}
                \left. \frac{\partial}{\partial u_1}\|\chi_\omega(x)u_1\|_{L^2(\Omega)}^2\right|_{u=u_1^*} 
                = 2 \langle \chi_\omega(x)^*\chi_\omega(x)u_1^*,h\rangle_{L^2(\Omega)},
            \end{align}
        Then,
            \begin{align*}
                J^{\prime}\left( u^*\right) h
                &=\rho \int_0^T \int_\Omega \ y_2^{\prime}( u^*)  y_2^*
                + \sigma_1\chi_\omega(x)^*\chi_\omega(x)u_1^*h + \sigma_2 u_2^*h\ d x dt.
            \end{align*}
        By letting 
        $  \Bar{u}^* = 
            \begin{bmatrix}
                \sigma_1 \chi_\omega(x)^*\chi_\omega(x) u_1^*\\
                \sigma_2 u_2^*
            \end{bmatrix},
        $
        we have 
            \begin{align*}    
                J^{\prime}\left( u^*\right) h
                & =\int_0^T
                \left\langle \rho D  y^*, D y'( u^*)  \right\rangle_{(L^2(\Omega))^2} d t
                +\int_0^T\left\langle  \bar{u}^*, h \right\rangle_{(L^2(\Omega))^2} \ d t 
                \\
                 & =\int_0^T
                \left\langle \rho D^*D  y^*, y'( u^*) \right\rangle_{(L^2(\Omega))^2} d t
                +\int_0^T\left\langle  \bar{u}^*, h \right\rangle_{(L^2(\Omega))^2} \ d t .
            \end{align*}
        Using the adjoint equation \eqref{Eq: Abstract_FDE_control_dual/adjoint}, we get
            \begin{align*}
             J^{\prime}\left( u^*\right) h
                & =\int_0^T
                \left\langle -\frac{\partial p}{\partial t}-Ap-B^* p, z \right\rangle_{(L^2(\Omega))^2} d t
                +\int_0^T\left\langle  \bar{u}^*, h \right\rangle_{(L^2(\Omega))^2} \ d t 
                \\
                & =z(0) p(0) + \int_0^T \left\langle z_t,  p \right\rangle_{(L^2(\Omega))^2} d t
                +\int_0^T \left\langle -A z -B z,  p \right\rangle_{(L^2(\Omega))^2} d t
                \\
                &\quad
                +\int_0^T \left\langle  \bar{u}^*, h \right\rangle_{(L^2(\Omega))^2} \ d t
                \\
                &=\int_0^T \left\langle z_t -A z -B z,  p \right\rangle_{(L^2(\Omega))^2} d t
                +\int_0^T \left\langle  \bar{u}^*, h \right\rangle_{(L^2(\Omega))^2} \ d t
                \\
                &=\int_0^T \left\langle Lp+ \bar{u}^*,  h \right\rangle_{(L^2(\Omega))^2} d t.
            \end{align*}
        
        Since $J$ is Gateaux differentiable at $u^*$ and $U_{ad}$ is convex, taking $h=v-u^*$ in the variational inequality $J'(u^*)(v-u^*) \ge 0$ for all $v\in U_{ad}$ implies $\bar u^*=-L p,$ or
            \begin{align*}
                u_1^* &= \frac{\chi_\omega y_1^* p_1}{\sigma_1}
                \quad\text{and}\quad
                u_2^* = \frac{y_2^* p_2}{\sigma_2}.
            \end{align*}
        Since $(u_1^*,u_2^*)\in U_{ad}$, we must have $0\le u_1^* \le U_1$ and $0\le u_2^*\le U_2$. Hence, the assertion follows.
    \end{proof}

\section{Numerical simulation}\label{NumerSimu}

\subsection{\textbf{Model parameters and regions}}

In this section, we give numerical simulations of the spread of contagious diseases, that is related to the different regions (subdomains) of choice where the supposedly vaccination strategy is implemented. We shall compare the disease propagation in some bounded domains, with and without control placed. To illustrate the numerical results, we use COVID-19 data from South Dakota, as considered by Forrest and Al-arydah \cite{Omar}. 
For implementations purposes, we rescale the governing equations by considering the proportion of each compartment with respect to the total population $N_{\mathrm{pop}}$. By writing $s = S/N_{pop}$ and $i = I/N_{pop}$, and substituting into the original system \eqref{Eq : FDE_control_system1}--\eqref{Eq : FDE_control_system2},
the governing equations of the primal problem becomes

\begin{equation}
    \begin{split}
        \partial_t s &= -d_1(-\Delta)^\sigma s - \beta\, s\, i - a_1 s + \tilde b - \chi_{\omega} u_1 s,\\
        \partial_t i &= -d_2(-\Delta)^\sigma i + \beta\, s\, i - a_2 i - u_2 i,
    \label{eq: primal_num}
    \end{split}
\end{equation}
where we have chosen $\psi(I)=\beta$ to be a (constant) transmission rate, and $\tilde b = \Pi/N_{pop}$, with $\Pi$ a constant recruitment/birth rate, whose values follow \cite{Omar}.

\subsubsection*{\textbf{Parameters}}
The diffusion coefficients $d_1=0.08$ and $d_2=0.02$ are chosen for the numerical experiment and are not estimated from a specific dataset. The choice $d_1>d_2$ is used to reflect the assumption that susceptible individuals move more freely, while infected individuals may have reduced mobility due to illness, self-isolation, or hospitalization.

In \cite{Omar}, as in ours, $a_1$ denotes the natural death rate, not the death rate caused by COVID-19. On the other hand, $a_2$ represents the total removal rate from the infected compartment, which consists of the recovery rate, the disease-induced death rate, and the natural death rate. We use the natural death rate and the recovery rate from \cite{Omar}. However, for the disease-induced death rate, we do not use the value assumed in \cite{Omar}, since the authors stated that the value $5.005\%$ was assumed rather than estimated from data. Instead, we use the value reported in \cite{carcione2020}, which is based on COVID-19 data, namely $\rho_u=0.0014$ per day. Therefore, combining the natural death rate and recovery rate from \cite{Omar} with the disease-induced death rate from \cite{carcione2020}, the value $a_2=0.0514256$ per day is used in our simulations.

For the upper bound of the vaccination control $U_1$, we take the U.S. COVID-19 vaccination campaign as a reference. According to the CDC \cite{cdc2021weekly}, by February 11, 2021, the peak (7-day average) number of COVID-19 vaccine doses administered in the United States was about $1.6$ million doses per day.
Since the U.S. population was approximately $331$ million as of
2020--2021 \cite{census2021}, this gives a daily vaccination rate of
$1.6\times10^6/(331\times10^6) \approx 0.0048$, or about $0.48\%$ of the population per day. Based on this reference value, we choose $U_1=0.01$ (or $1\%$) as the upper bound for the vaccination control used in our simulation. 
This value is slightly higher than the observed peak to allow an aggressive but still reasonable daily vaccination capacity.

For the upper bound of the treatment control, we consider $U_2=0.03$. Although this gives a larger proportion bound than the vaccination control, this choice is deemed reasonable because treatment is applied only to the infected population. This population is typically much smaller than the susceptible population targeted by vaccination. 

The cost weights $\sigma_1$, $\sigma_2$, and $\rho$ in the objective functional are chosen based on the relative importance of the quantities they represent, rather than as a real monetary costs. In particular, the choice $\sigma_2>\sigma_1$ is motivated by the observation that treatment or hospitalization of the infected is more costly than vaccination. For example, an average COVID-19 vaccine dose costs around \$40 \cite{aspe2022}, while the average cost of a COVID-19 hospitalization is around \$20,000 \cite{kff2022}. This gives a very high vaccination versus treatment cost ratio.

In this paper, we take $\sigma_1=5$ and $\sigma_2=400$, to have $\sigma_2:\sigma_1=80:1$ ratio. The infection weight is chosen as $\rho=50$, which is larger than $\sigma_1$. This choice reflects the fact that the cost associated with having infected individuals, such as loss of productivity, additional pressure on the healthcare system, and the risk of further transmission, is much larger than the cost of vaccination alone. We also see that this way, reducing the infected population is given higher priority than minimizing the vaccination effort. 
The summary of the parameters and initial values are given in Table \ref{table_parameters}

\begin{table}[htbp!]
\centering
\caption{Initial conditions and parameter values}
    \begin{tabular}{c c l l}
        \hline
        \hline
        \textbf{Notation} & \textbf{Value} & \textbf{Description} & \textbf{Source} \\ 
        \hline
        \hline
            $N_{pop}$    & $864{,}822$            & Total population scale            & \cite{Omar} \\
            $s_{total}$       & $0.4984$                & Initial susceptible fraction       & \cite{Omar} \\
            $i_{total}$       & $4.4671\times10^{-4}$   & Initial infected fraction          & \cite{Omar} \\
            $\beta_0$    & $0.25$                  & Transmission rate (day$^{-1}$)     & \cite{Omar} \\
            $a_1$        & $2.555\times10^{-5}$    & Death rate, natural causes (day$^{-1}$) & \cite{Omar} \\
            $a_2$        & $0.0514256$             & Removal rate of infected (day$^{-1}$) & \cite{Omar}, \cite{carcione2020} \\
            $b$          & $4.25475\times10^{-5}$  & Recruitment rate (day$^{-1}$)      & \cite{Omar} \& rescaled \\
            $d_1$        & $0.08$                   & Diffusion coeff., susceptible      & Assumed \\
            $d_2$        & $0.02$                   & Diffusion coeff., infected         & Assumed \\
            $\sigma_1$   & $5$                     & Vaccination cost weight            & Assumed \\
            $\sigma_2$   & $400$                   & Treatment cost weight              & Assumed \\
            $\rho$       & $50$                    & Infection cost weight              & Assumed \\
            $U_1$        & $0.01$                  & Vaccination upper bound (\%/day)   & \cite{cdc2021weekly} \\
            $U_2$        & $0.03$                  & Treatment upper bound (\%/day)     & Assumed \\
            $T$          & $600$                   & Time horizon (days)                & Assumed \\
        \hline
        \hline
    \end{tabular}
    \label{table_parameters}
\end{table}

\subsubsection*{\textbf{Regions choice}}

We consider the spatial domain $\Omega=[0,1]\times[0,1]$, discretized on a uniform $N\times N=64\times 64$ grid. 
We consider two different geometries for the regional control to illustrate the disease progression. The first region is the circular geometry placed adjacent to the boundary, and the second one is the rectangular geometry, where the initial infected region is located at the centre of the domain (see Table \ref{table:regions} for the exact definitions and areas). These two geometries are used to compare between an outbreak concentrated near the boundary and an outbreak which is interior concentrated.
For both cases, the initial susceptible and infected populations are taken to be non-uniformly distributed to mimic the real world scenario. In this settings, $40\%$ of the total initial  susceptible fraction $s_{total}$ and $60\%$ of the total initial infected fraction $i_{total}$ are placed inside the corresponding region $\omega_i$ ($i=1,2$). The remaining $60\%$ of $s_{total}$ and $40\%$ of $i_{total}$ are distributed over the complement $\Omega\setminus \omega_i$ ($i=1,2$).

The vaccination control region $\chi_{\omega_i}(x)$ ($i=1,2$) is chosen to be the same geometry as the initial infected region plus a slight extension outward by a buffer (of six grid cells of the discretization). This means that the control is concentrated around the outbreak region, but is not restricted exactly to it. 

\begin{table}[htbp!]
    \centering
    \caption{Initial condition geometries, control regions, and population distributions}
        \begin{tabular}{l l c}
            \hline
            \hline
            \multicolumn{3}{c}{\textbf{Outbreak geometry $\omega$}} \\
            \textbf{Name} & \textbf{Region} & \textbf{Area $|\omega|$} \\
            \hline
            Circular-corner   & $\{(x,y)\in\Omega : x^2+(y-1)^2 \le (0.5)^2\}$            & $\approx 0.196$ \\
            Rectangular       & $\{(x,y)\in\Omega : |x-0.5|\le 0.2,\ |y-0.5|\le 0.2\}$      & $\approx 0.160$ \\
            Circular-middle   & $\{(x,y)\in\Omega : (x-0.5)^2+(y-0.5)^2 \le (0.3)^2\}$     & $\approx 0.283$ \\
            \hline
            \multicolumn{3}{c}{\textbf{Control region}} \\
            Circular-corner   & $\{(x,y)\in\Omega : x^2+(y-1)^2 \le (0.5+6dx)^2\}$            & $\approx 0.277$ \\
            Rectangular       & $\{(x,y)\in\Omega : |x-0.5|\le 0.2+6dx,\ |y-0.5|\le 0.2+6dx\}$ & $\approx 0.345$ \\
            Circular-middle   & $\{(x,y)\in\Omega : (x-0.5)^2+(y-0.5)^2 \le (0.3+6dx)^2\}$     & $\approx 0.487$ \\
            \hline
            \\
            \multicolumn{3}{c}{\textbf{Initial population distribution}} \\
            \textbf{Variable} & \textbf{Population distribution} & 
            \\
            \hline
            $s_0(x,y)$ & $40\%\, s_{total}$ (for $(x,y)\in\omega$) $+\ 60\%\, s_{total}$ (for $(x,y)\in\Omega\setminus\omega$) & 
            \\
            $i_0(x,y)$ & $60\%\, i_{total}$ (for $(x,y)\in\omega$) $+\ 40\%\, i_{total}$ (for $(x,y)\in\Omega\setminus\omega$) &\\
            \hline
            \hline
            \label{table:regions}
        \end{tabular}
    \end{table}
The percentages in the table indicate how the total initial fractions $s_{\mathrm{total}}$ and $i_{\mathrm{total}}$ are distributed between $\omega$ and $\Omega\setminus\omega$. To obtain the initial density functions $s_0(x,y)$ and $i_0(x,y)$, each portion is divided by the area of the region where it is placed, namely
\[
    s_0(x,y)
    =
    \begin{cases}
        \dfrac{0.40\,s_{\mathrm{total}}}{|\omega|},
        & (x,y)\in\omega,\\[1ex]
        \dfrac{0.60\,s_{\mathrm{total}}}{|\Omega\setminus\omega|},
        & (x,y)\in\Omega\setminus\omega,
    \end{cases}
\qquad
\text{and}
\qquad
    i_0(x,y)
    =
    \begin{cases}
        \dfrac{0.60\,i_{\mathrm{total}}}{|\omega|},
        & (x,y)\in\omega,\\[1ex]
        \dfrac{0.40\,i_{\mathrm{total}}}{|\Omega\setminus\omega|},
        & (x,y)\in\Omega\setminus\omega.
    \end{cases}
\]
This ensures that $\displaystyle\int_\Omega s_0(x,y)\,dxdy=s_{\mathrm{total}}$ and $\displaystyle\int_\Omega i_0(x,y)\,dxdy=i_{\mathrm{total}}$.

\subsection{{Numerical scheme}}
\subsubsection*{\textbf{Forward-backward scheme}}
To find the optimal control, we use a forward--backward iterative algorithm. At each iteration, we move forward in time by solving the (rescaled) primal problem \eqref{eq: primal_num}, then use the resulting state to move backward in time by solving the dual problem. Notice that our (rescaled) primal problem \eqref{eq: primal_num} can be written in the abstract form \eqref{eq : F1 F2 RHS}-\eqref{Eq: Abstract_FDE_control}, so that it corresponds to the similar adjoint equation \eqref{Eq: Abstract_FDE_control_dual/adjoint}. With our parameters choice, the corresponding dual equation reads
    \begin{align}
        \begin{cases}
            -\dfrac{\partial p_1}{\partial t} &= -d_1(-\Delta)^\sigma p_1 
                + \left(-i\beta_0 - a_1 - \chi u_1\right) p_1 + (i\beta_0)\, p_2, 
                \\[6pt]
            -\dfrac{\partial p_2}{\partial t} &= -d_2(-\Delta)^\sigma p_2 
                + \left(-s\beta_0\right) p_1 + \left(s\beta_0 - a_2 - u_2\right) p_2 
                + \rho i,
        \end{cases}
    \label{eq:dual_num}
    \end{align}
with terminal condition $p_1(x,y,T) = p_2(x,y,T) = 0$. 

The control is then updated using both the primal and dual solutions, and this forward--backward pass is repeated until the control barely changes between iterations (i.e., the change falls below a chosen tolerance).
From the necessary optimality condition, we seen that the solution of the optimal control problem is characterized by \eqref{eq: u* (optimal)}. The implementation of an optimal control strategy of the problem \eqref{Eq: Objective_Functional} related to the system \eqref{eq: primal_num} can be achieved by the following iterative scheme
    \begin{align}
        \begin{cases}
            u_1^{n} (x,y,t) &= \min\left(U_1,\ \max\left(0,\
                \dfrac{\chi_\omega(x,y)\, s^n(x,y,t)\, p_1^n(x,y,t)}{\sigma_1}\right)\right), \\[6pt]
            u_2^{n} (x,y,t) &= \min\left(U_2,\ \max\left(0,\
                \dfrac{i^n(x,y,t)\, p_2^n(x,y,t)}{\sigma_2}\right)\right),
        \end{cases}
        \label{eq:optimal_u}
    \end{align}
where $(s^n, i^n)$ is the solution of the primal system \eqref{eq: primal_num} associated with the control $u^n = (u_1^n, u_2^n)$ at the $n$-th iteration, and $(p_1^n, p_2^n)$ is the solution of the dual(adjoint) system \eqref{eq:dual_num} computed using $(s^n, i^n)$ and $u^n$.
After computing the candidate controls $u^n$ from \eqref{eq:optimal_u}, we apply a relaxation step. Following the forward--backward sweep procedure described in Lenhart and Workman \cite{LenhartWorkman2007}, we set  $u_j^{n+1} = \theta u_j^{n} + (1-\theta)u_j^{n-1}$ for $j=1,2$, where $\theta\in(0,1]$ is a relaxation parameter.
This step is a standard modification to the forward--backward sweep method, used to prevent oscillatory (non-convergent) behavior in the control updates between iterations \cite{LenhartWorkman2007}. In our case, we choose $\theta=0.25$. This forward--backward pass, including the relaxation step, is repeated until the control barely changes between successive iterations, i.e.,
    \begin{align}
        \max\left(\left\|u_1^{n+1} - u_1^{n}\right\|_\infty,\
                  \left\|u_2^{n+1} - u_2^{n}\right\|_\infty\right) < \text{tol}.
    \end{align}

\subsubsection*{\textbf{Solving fractional reaction-diffusion system via the Fourier Spectral Method}}
The algorithm to solve the problem is as given in the Algorithm \ref{Algorithm}. In solving both primal system \eqref{eq : FDE yn} and adjoint system \eqref{Eq: Abstract_FDE_control_dual/adjoint}, we adopt the Fourier spectral method numerical scheme given in \cite{BuenoOrovioKayBurrage}, which we specify here for the sake of clarity.

Consider the system
    \begin{align}
        \begin{cases}
            \dfrac{dv}{dt} = Av + G(t,v(t)), \qquad t\in[0,T], \\
            v(0) = v^0,
        \end{cases}
    \end{align}
where $v = (v_1,v_2)$. In our application, $v$ may represent either the state variable $(s,i)$ or the dual variable $(p_1,p_2)$, whereas $G$ is the corresponding right-hand side of the primal problem \eqref{eq: primal_num}, or the dual problem \eqref{eq:dual_num}, respectively.  We consider the time interval discretization $t_0 < t_1 < \cdots < t_n < \cdots < T$, with $t_n = t_0 + n\Delta t = t_{n-1}+\Delta t$, and write $v^n = v(x,y,t^n)$. Using backward Euler in time, we have 
    \begin{align}
        \frac{v^{n+1}-v^n}{\Delta t} = Av^{n+1} + G(t^{n+1},v^{n+1}).
    \end{align}
Since $G$ is, in general, nonlinear in $v$, we cannot directly apply the discrete Fourier transform.  Therefore, we adopt the fixed-point iteration as mentioned in \cite{BuenoOrovioKayBurrage} as follows: given $v^n$ at the $n-th$ timestep, we initialize $v^{n+1,0} := v^n$ and for $m = 1, \ldots, M$ we solve:
    \begin{align}
        \frac{v^{n+1,m}-v^n}{\Delta t} = Av^{n+1,m} + G(t^{n+1},v^{n+1,m-1}),
        \label{eq:scheme_backward_fixed}.
    \end{align}
Here, the nonlinear term $G$ is evaluated at the previous fixed-point iterate $v^{n+1,m-1}$, which is known. In our case, we choose $M=3$. Applying the discrete Fourier transform to \eqref{eq:scheme_backward_fixed}, we obtain
    \begin{align}
        \hat{v}_{ij}^{n+1,m} = \frac{1}{1+d_i\lambda_j^\sigma \Delta t}
        \left[\hat{v}_{ij}^{n} + \Delta t\,\hat{G}_{ij}\bigl(v^{n+1,m-1},t^{n+1}\bigr)\right],
        \qquad i=1,2,
    \end{align}
where $\hat{(\cdot)}$ denotes the discrete Fourier transform and $\hat{G}_{ij}$ is obtained by evaluating $G_i$ at the grid values of $v^{n+1,m-1}$ and applying the discrete Fourier transform. The solution at each time step is recovered by the inverse discrete Fourier transform.

\begin{algorithm}
\caption{An optimal control approach}
        \label{Algorithm}
            \begin{algorithmic}
                \State \textbf{Step 1:} Initialize system data
                        \State Set the initial state $s^0, i^0$.
                        \State Set the initial control $(u_1^0,u_2^0) = (0,0)$.
                        \State Set the control region $\chi_\omega$, parameters
                            $d_1,d_2,a_1,a_2,b,\beta_0,\sigma_1,\sigma_2,\rho$,
                            final time $T$, relaxation parameter $\theta$, and
                            tolerance $\mathrm{tol}$.
                \State \textbf{Step 2:}
                \While{$\max\left(\|u_1^{n+1}-u_1^{n}\|_\infty,\
                       \|u_2^{n+1}-u_2^{n}\|_\infty\right) > \mathrm{tol}$}
                        \State Obtain $s^n, i^n$ by solving the primal system
                            \eqref{eq: primal_num} using the Fourier spectral
                            method.
                        \State Obtain $p_1^n, p_2^n$ by solving the dual system
                            \eqref{eq:dual_num} using the Fourier spectral
                            method.
                        \State Compute the candidate control via
                            \eqref{eq:optimal_u} and do the relaxation step to update the control.
                \EndWhile
            \end{algorithmic}
        \end{algorithm}
        All numerical simulations were performed using MATLAB R2025a.
        
        \subsection{{Results and discussion}}
            In this subsection, we present the results obtained by applying the forward--backward algorithm in Algorithm \ref{Algorithm} for $\sigma \in \{1.00,\ 0.80,\ 0.55\}$. We focus on the circular-corner and rectangular geometries summarized in Table \ref{table_parameters}. 
            
            For each value of $\sigma$, we compare the uncontrolled system \eqref{eq: primal_num}, where $u_1=u_2=0$, with the fully controlled system. The comparison is made in terms of the infected compartment $\int_\Omega i(t,x)\,dx$, the susceptible compartment $\int_\Omega s(t,x)\,dx$, and the optimal control profiles $u_1^*(t)$ and $u_2^*(t)$. Our goal is to see how the fractional order affects the spread of the disease and, as a result, the optimal control strategy.

        \subsubsection*{\textbf{Controlled versus uncontrolled dynamics}}
            The first one that we noticed is that, as we can see from Figure \ref{fig:four comparison} the model with the controlled case has a significantly lower peak of the total infected proportion (represented by the $L^1$ norm $\|i\|_{L^1(\Omega)}$) than the model without control placed. We also see from Figure \ref{fig:four comparison} that, for both geometries, the peak of the infected proportion in the uncontrolled case reaches up to approximately $\|i\|_{L^1(\Omega)} \approx 0.11$, which translates into approximately $0.11 \times N_{\mathrm{pop}} \approx 95{,}400$ individuals in our real-world setting. In the controlled case, the growth of the disease is suppressed, and the peak is reduced to roughly $\|i\|_{L^1(\Omega)} \approx 0.0455$ for the circular-corner IC and $\|i\|_{L^1(\Omega)} \approx 0.0377$ for the rectangular IC, which translate into approximately $39{,}300$ and $32{,}600$ individuals, respectively, which is a reduction of about $56{,}000$ and $63{,}000$ infected individuals at the peak, that corresponds to roughly $59\%$ and $66\%$ reductions relative to the uncontrolled peak.

            We also notice that there is a crossover phenomenon in large time, approximately around $t^* \approx 170$ for the circular-corner IC and $t^* \approx 178$ for the rectangular IC, where the number of infected individuals in the controlled system becomes slightly larger than in the uncontrolled case. This phenomenon is not new; it has also been observed, for instance, in \cite{Ketcheson2021}, and is not hard to explain. In the uncontrolled case, the disease spreads much faster, so that during the early phase many more individuals become infected than in the controlled case; consequently, the susceptible population also decreases more sharply in the uncontrolled case, for both geometries, as seen on the left-hand side of Figure \ref{fig:four comparison}. After the peak, there are comparatively few susceptibles left to infect in the uncontrolled case, whereas in the controlled case a larger susceptible population remains available, allowing further infections to occur. As a result, at later times the infected population in the controlled case may exceed that of the uncontrolled case, although only by an insignificant amount, after the disease propagation has been contained. This reduced peak is nonetheless of great practical importance, as it allows healthcare resources to be allocated more effectively, which is always a critical concern, since many countries experienced a severe strain even collapses on their healthcare systems during the COVID-19 pandemic.
            Figure \ref{fig:control_profiles_per_sigma} shows the corresponding optimal control rate profiles $\langle u_1\rangle_x(t)$ (vaccination) and $\langle u_2\rangle_x(t)$ (treatment) that produce this peak reduction, for both geometries and all three values of $\sigma$.
            
            \begin{figure}
                \centering
                \includegraphics[width=.8\linewidth]{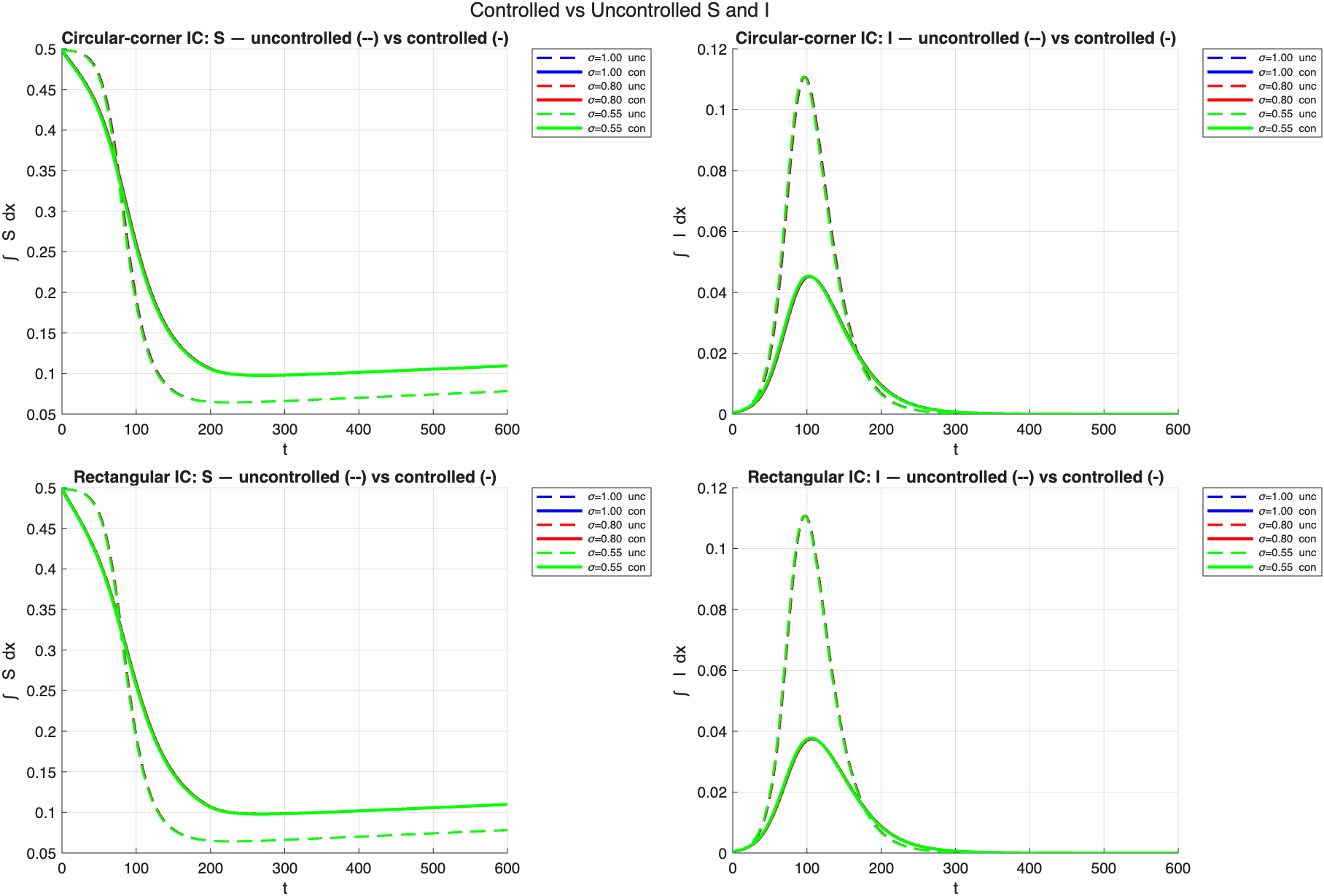}
                \caption{$L^1(\Omega)$ norms $\|s\|_{L^1(\Omega)}$ and
                    $\|i\|_{L^1(\Omega)}$, controlled vs.\ uncontrolled, for
                    $\sigma=1.00,\ 0.80,\ 0.55$: circular-corner (top) and rectangular
                    (bottom).}
                \label{fig:four comparison}
            \end{figure}
        
            \begin{figure}[htbp!]
                \centering
                \includegraphics[width=0.45\textwidth]{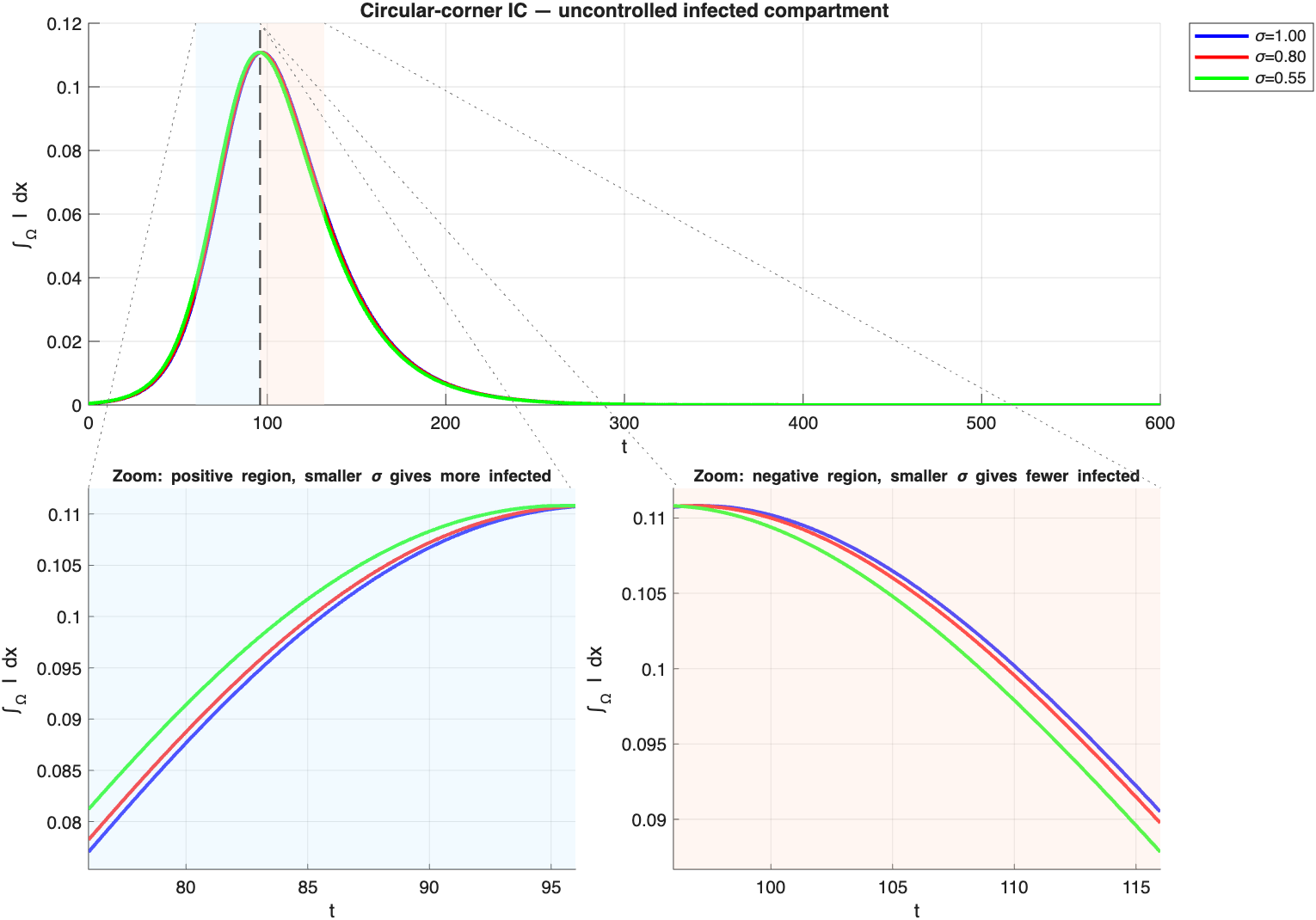}
                \includegraphics[width=0.45\textwidth]{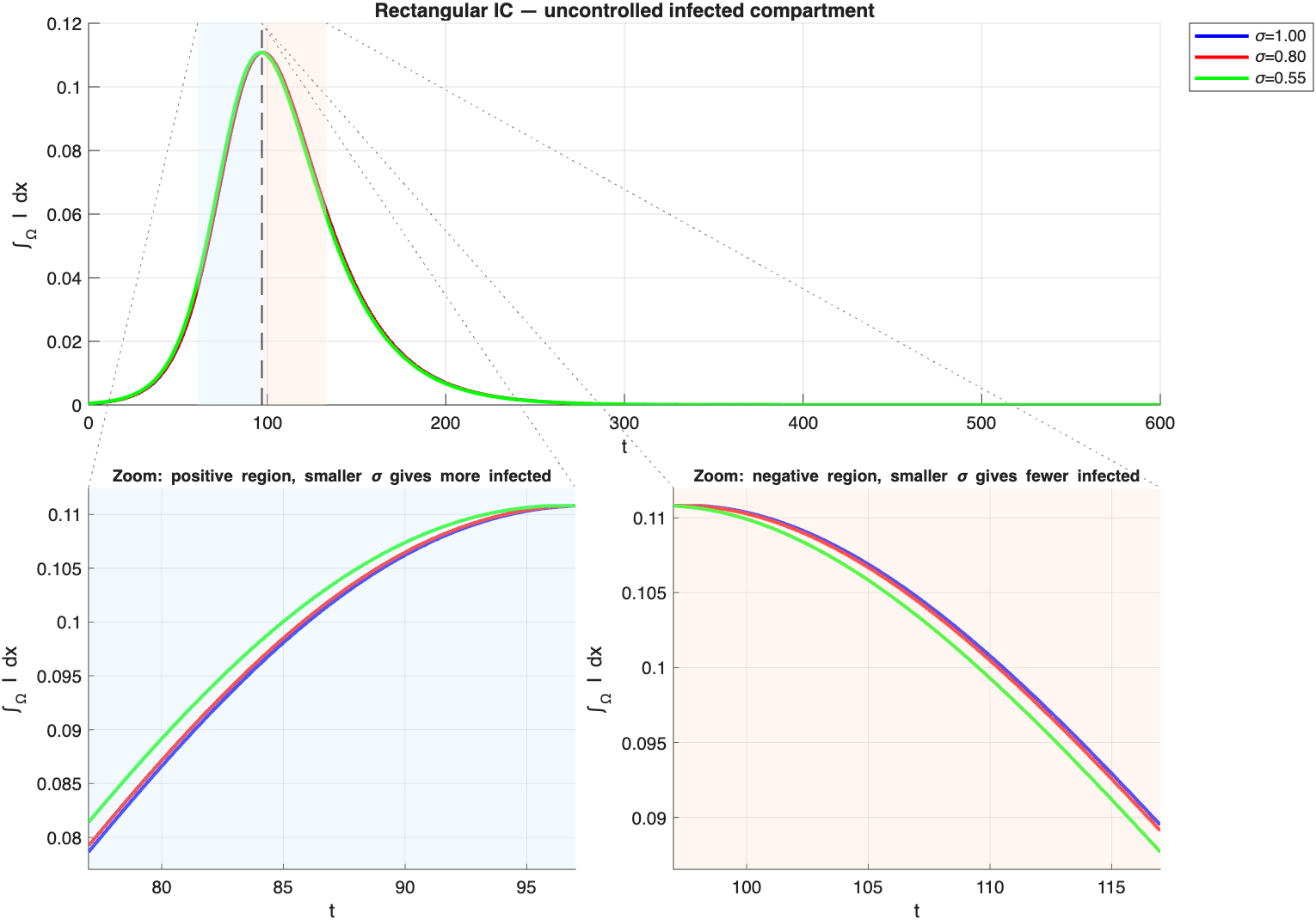}
                \\
                \vspace{0.4cm}
                \includegraphics[width=0.45\textwidth]{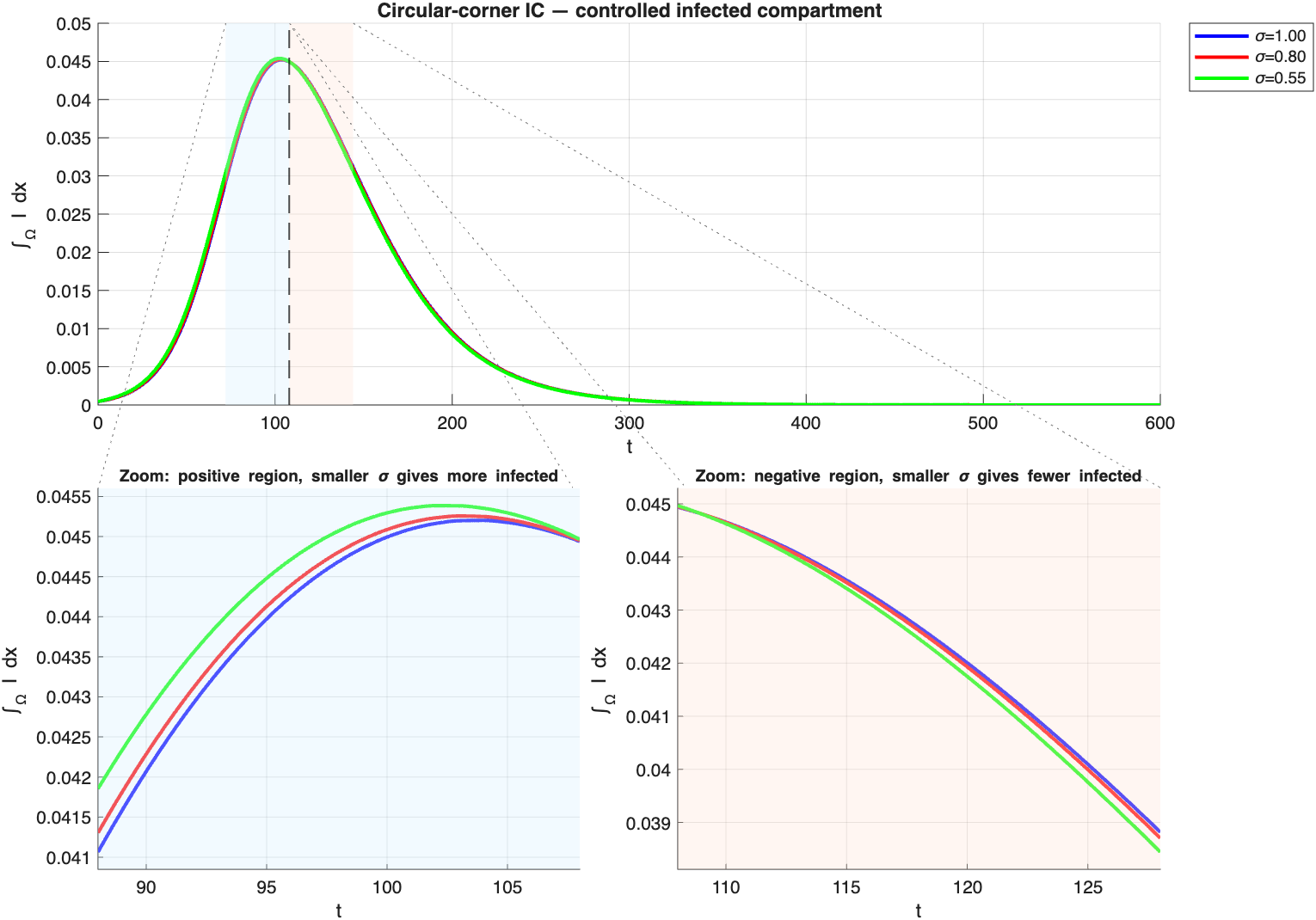}
                \includegraphics[width=0.45\textwidth]{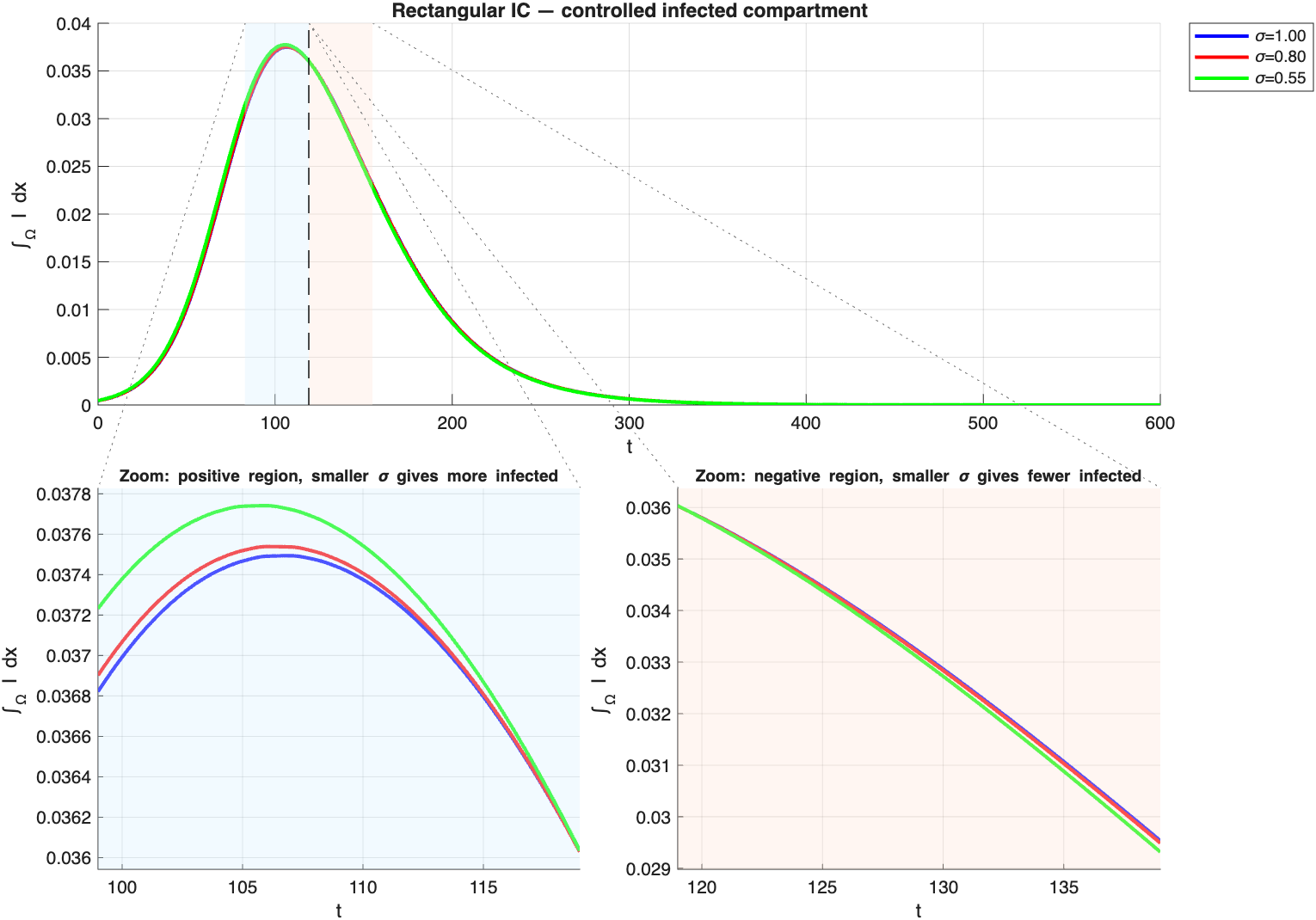}
                \caption{Infected compartment $\int_\Omega i\,dx$ near the
                fractional-order crossover, for $\sigma=1.00,\ 0.80,\ 0.55$. Top
                row: uncontrolled case (circular-corner IC, left; rectangular IC,
                right). Bottom row: controlled case (circular-corner IC, left;
                rectangular IC, right). For each panel, the left zoomed inset shows
                the region where smaller $\sigma$ gives more infected, and the
                right zoomed inset shows the region after the crossover, where
                smaller $\sigma$ gives fewer infected.}
                \label{fig:zoom_panels_combined}
            \end{figure}

            \begin{figure}[htbp!]
                \centering
                \includegraphics[width=0.45\textwidth]{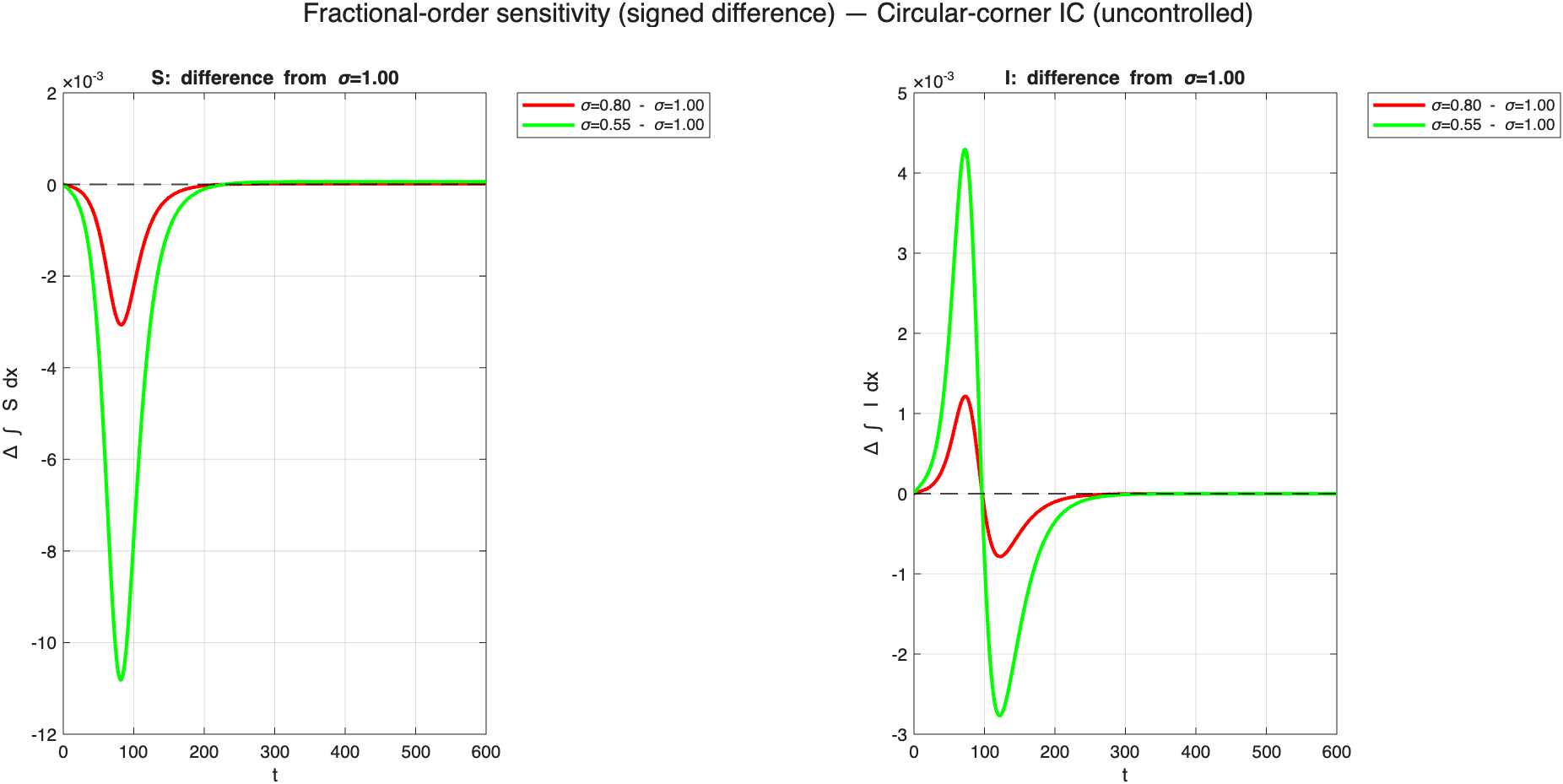}
                \includegraphics[width=0.45\textwidth]{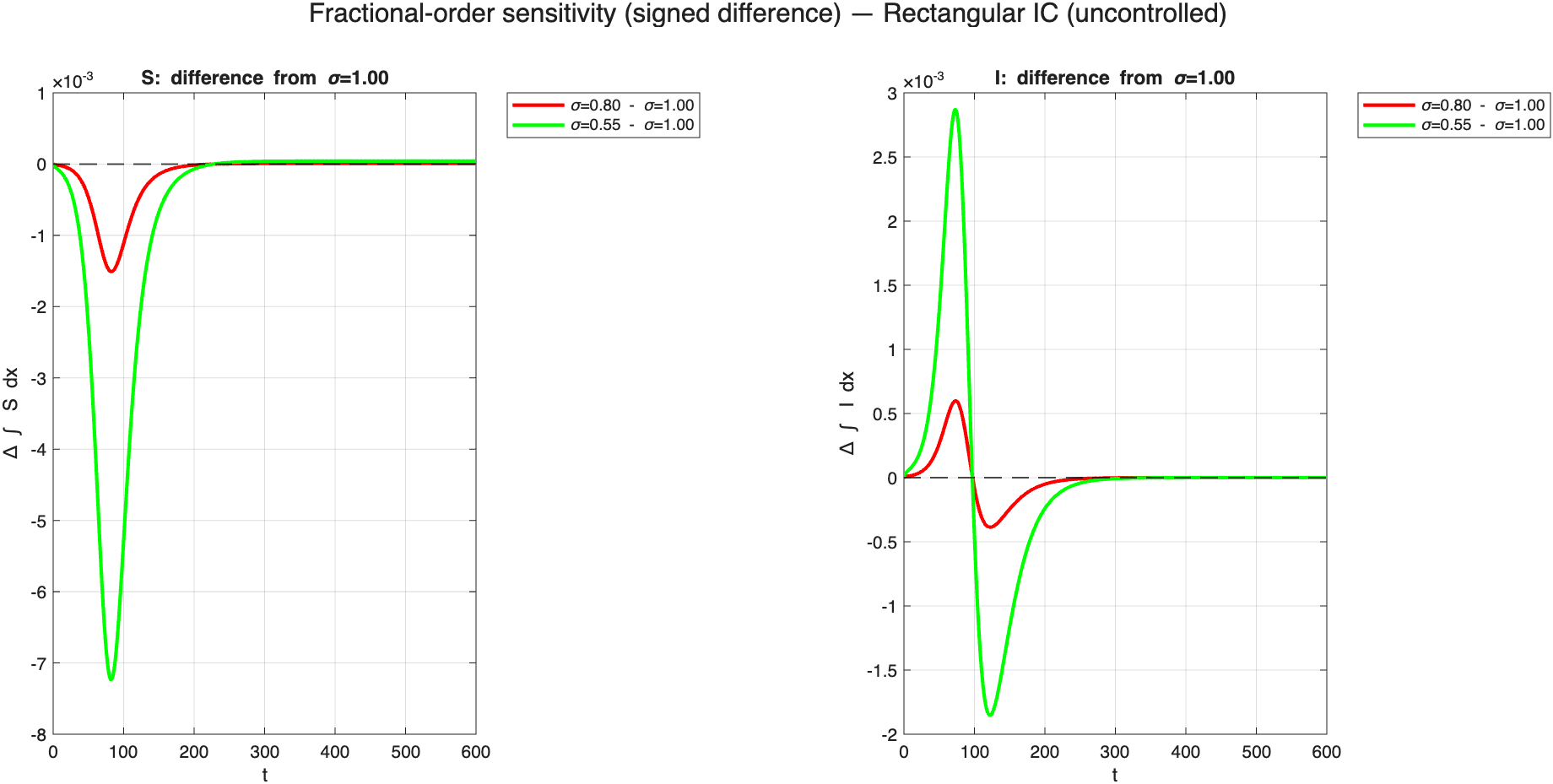}
                \\
                \vspace{0.4cm}
                \includegraphics[width=0.45\textwidth]{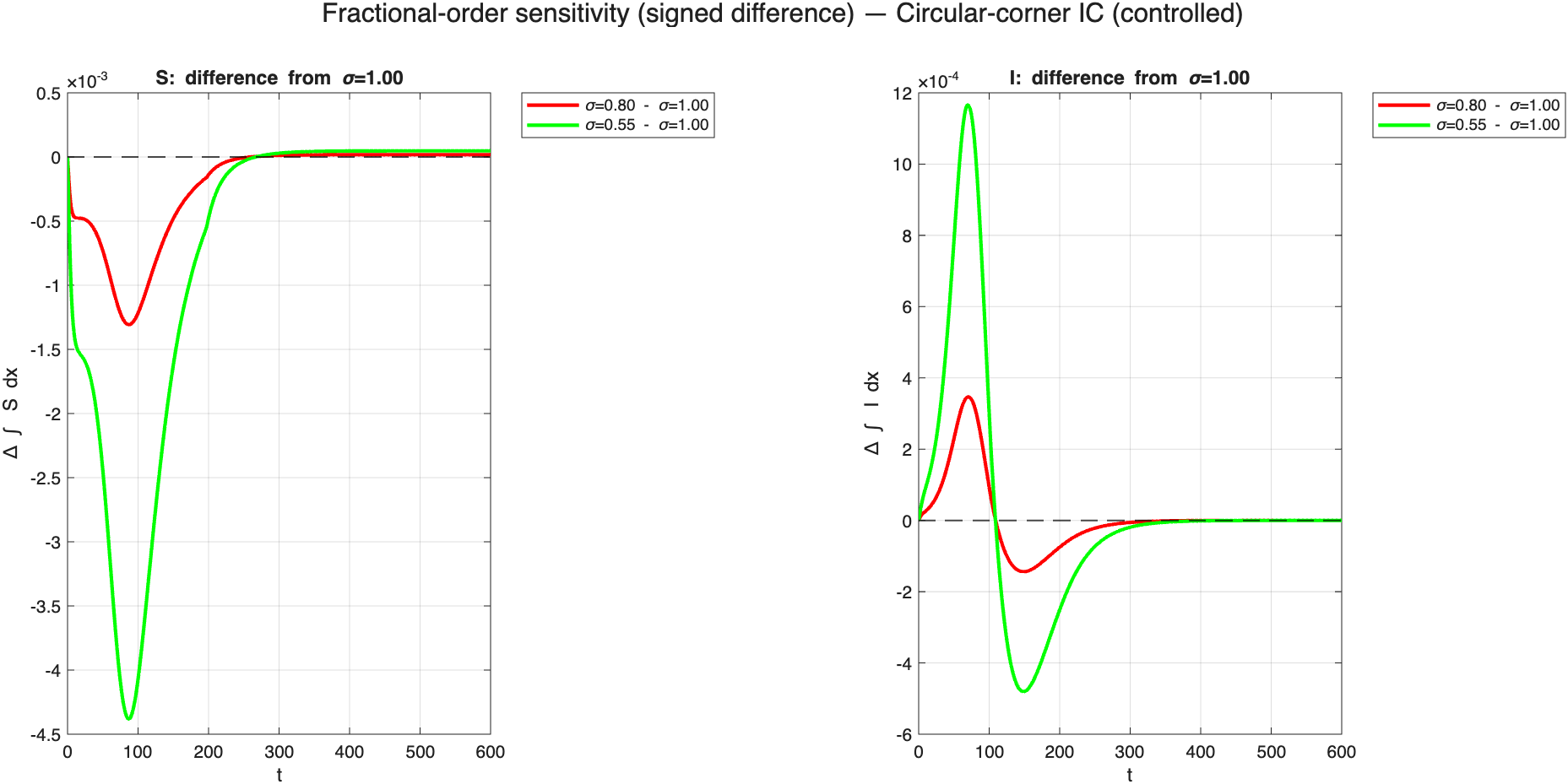}
                \includegraphics[width=0.45\textwidth]{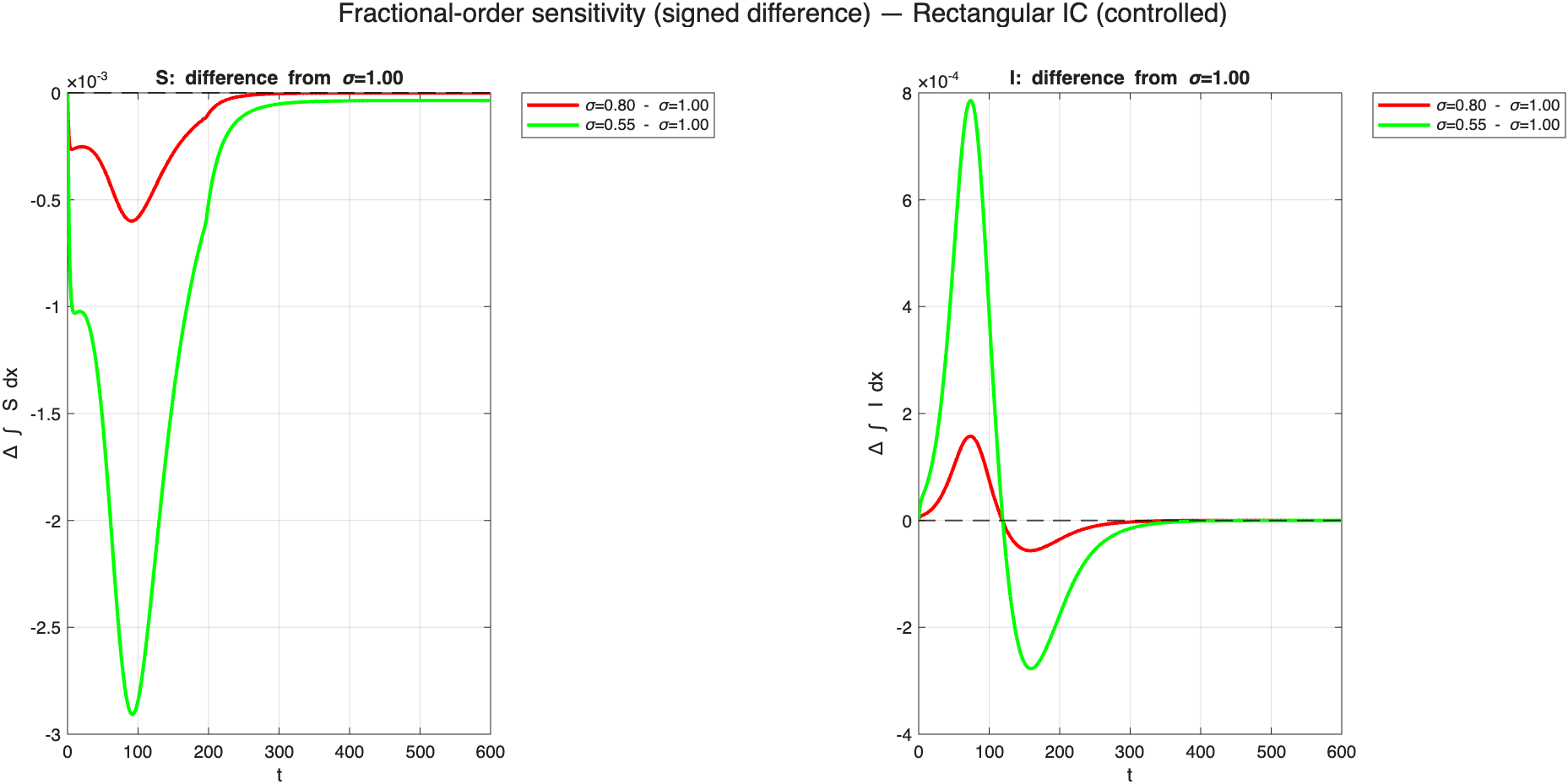}
                \caption{Signed difference of $\|S\|_{L^1(\Omega)}$ and
                $\|I\|_{L^1(\Omega)}$ for $\sigma\in\{0.55,0.80\}$ relative to
                $\sigma=1.00$. Top row: uncontrolled case (circular-corner IC, left;
                rectangular IC, right). Bottom row: controlled case
                (circular-corner IC, left; rectangular IC, right).}
                \label{fig:sensitivity_sigma1_combined}
            \end{figure}

        \subsubsection*{\textbf{Fractional order effects on the infected compartment}}
            The second thing worth noticing is that the model with fractional order $\sigma<1$ induces a known \emph{skewed} temporal behavior in the infected compartment. As seen in Figure \ref{fig:zoom_panels_combined}, for both the circular-corner and rectangular geometries, the infected compartment initially follows the ordering $\sigma=0.55 > \sigma=0.80 > \sigma=1.00$ during the early growth phase of the epidemic, with smaller fractional orders producing a larger infected population. However, there exists some time $t>0$ where this ordering reverses, so that the infected compartment spatial mean follows the ordering $\sigma=0.55 < \sigma=0.80 < \sigma=1.00$ for the remainder of the outbreak. However, the gap between the infected trajectories after the crossover time is noticeably smaller in magnitude than the gap of the pre-crossover infected trajectories, so the trajectories end up much closer together than they were before the crossover. This skewness behavior is also confirmed by the signed difference shown in the bottom row of Figure \ref{fig:sensitivity_sigma1_combined} (controlled case): for both geometries and for both pairs ($\sigma=0.80$ vs.\ $\sigma=1.00$, and $\sigma=0.55$ vs. $\sigma=1.00$), the positive peak of $\Delta\|i\|_{L^1(\Omega)}$ is larger in magnitude than the negative dip that follows it. As expected, this effect is more visible for $\sigma=0.55$ than for $\sigma=0.80$. We also notice that this skewness behavior also appears in the uncontrolled case, as illustrated in the top row of Figure \ref{fig:sensitivity_sigma1_combined}, where there is a change between positive and negative peak. However, we see that in the uncontrolled case, the skewness is still there but less dramatic.
            In particular, the positive peak of the difference of  $\|i\|_{L^1(\Omega)}$ before the crossover time and the negative dip after the crossover time have almost comparable magnitude. Meanwhile, in the controlled case, the negative dip after the crossover is much smaller than the positive peak before the crossover.
            
            The skewness result of the infected compartments, coming from the model with fractional Laplacian has an interesting feature for further observation. As $\sigma$ decreases from $1$, the model with $(-\Delta)^\sigma$ becomes increasingly heavy-tailed, which allows (a small portion of) individuals to travel directly over long spatial distances rather than commuting only to their immediate neighborhood, as in the classical (local) diffusion case $\sigma=1$.

            This results in initially higher infected in the beginning for smaller order fractional $\sigma$ but eventually lower infected towards the end, which happens simply because the high mobility of the infectives, which are allowed to move directly to remote sites, which allow them to reach susceptible individuals located far from the initial outbreak region much sooner than would be possible under classical diffusion model with $\sigma=1$. This produces a sharper and earlier rise in the infected compartment, and explains why the number of infected people with $\sigma=0.55$ exceed the one of $\sigma=0.80$ which also the one of $\sigma=1.00$ during this phase.
            
            Once the susceptible population has been substantially depleted, the epidemic under smaller $\sigma$ has comparatively fewer opportunities for further transmission than the epidemic under higher $\sigma$. This is why the trajectory starts changing its course, and the number of infected compartments related to the model with lower order sigma started to be smaller than the number of infected corresponding to the higher order sigma. This phenomenon is what we see in both uncontrolled case, and the controlled case .

            \begin{figure}[htbp!]
                \centering
                \includegraphics[width=0.4\textwidth]{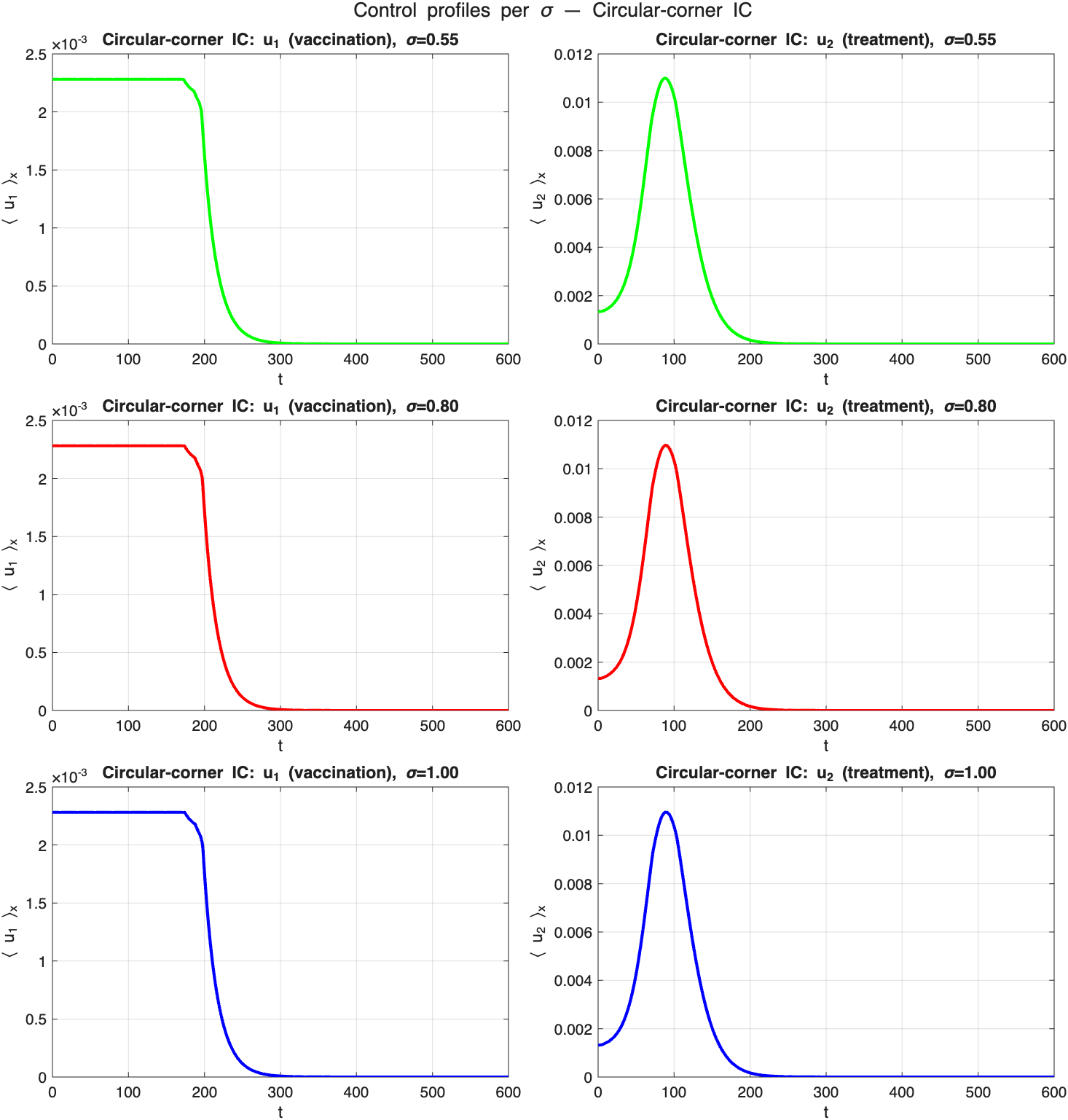}
                \includegraphics[width=0.5\textwidth]{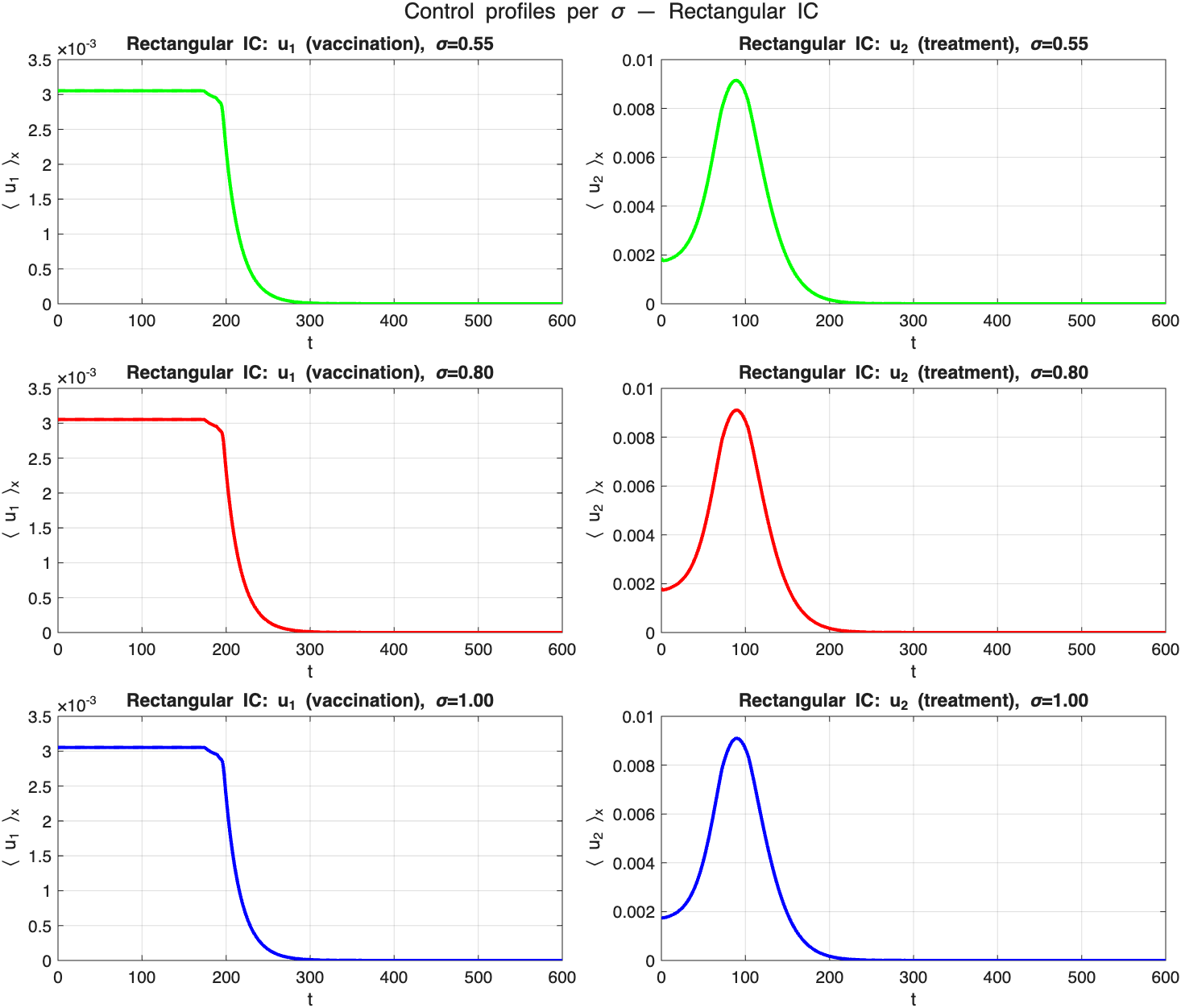}
                \caption{Optimal control rate profiles $\langle u_1\rangle_x(t)$
                (vaccination) and $\langle u_2\rangle_x(t)$ (treatment), shown
                separately for $\sigma=0.55,\ 0.80,\ 1.00$ (top to bottom). Left:
                circular-corner IC. Right: rectangular IC.}
                \label{fig:control_profiles_per_sigma}
            \end{figure}
        
        \subsubsection*{\textbf{Mobility effect and regional vaccination}}
            Now, for the controlled case, what we see is even more interesting. The skewness behavior becomes more visible in the trajectory of the solution. In the early phase of the outbreak, the long-range mobility associated with smaller $\sigma$ allows the infection to reach susceptible individuals located farther away from the initial outbreak region. This is the same mechanism that we observe in the uncontrolled case. Hence, we still obtain the same early-phase ordering $\sigma=0.55>\sigma=0.80>\sigma=1.00$ in the relative sense, even though the controlled trajectories are already much lower than the uncontrolled ones in terms of their absolute values.
            
            At later times, however, this same mobility also affects how the susceptible population interacts with the vaccination region $\chi_\omega$ in \eqref{eq: primal_num}. Although the vaccination region is fixed in space, when $\sigma$ is smaller, susceptible individuals outside $\chi_\omega$ can move into the vaccination region more easily. In contrast, for the local diffusion case $\sigma=1.00$, those susceptible individuals would be less affected by a spatially fixed intervention if they stay far from $\chi_\omega$. Therefore, through the term $-\chi_\omega u_1 s$, a larger part of the susceptible population can be exposed to vaccination when $\sigma$ is smaller.
            
            Over the active period of vaccination, more susceptible individuals are removed when $\sigma$ is smaller. This effect becomes more visible in the long run. Since the mechanism of vaccination is the same for all values of the fractional order $\sigma$, the differences between the infected trajectories after the crossover time have a smaller negative dip compared to the uncontrolled case. On the other hand, the gap before the crossover time is mainly driven by how fast the infection itself spreads to susceptible individuals located far from the initial outbreak region, due to the fractional order $\sigma$. At this stage, the susceptible population is still quite large, and only a small portion of the population has moved to distant sites. Therefore, the transmission-driven difference between the fractional orders is still dominant, while the effect of vaccination is still comparatively small and gradual. This is why the control reduces the post-crossover gap more strongly than the pre-crossover gap. As a result, we obtain a more skewed relative profile in the controlled case, even though the overall infection level is already much lower in absolute value throughout the whole time interval.

        \subsubsection*{\textbf{Discussion and policy recommendation}}
            The results in the previous subsections show that the fractional order $\sigma$ does not only affect how the epidemic spreads without control, but also affects the form of the optimal vaccination strategy once the control is introduced. In practice, vaccines usually become available only after some time, when the local infection is already visible. Therefore, in this setting, our results are more relevant for understanding how a government should distribute the available vaccination effort once the vaccine becomes available.
        
            Since higher mobility, represented by smaller $\sigma$, allows the infection to reach distant susceptible populations more quickly, our results suggest that vaccination should be applied earlier and more strongly for a highly mobile population than for a population with more local diffusion. This also gives a possible implication when the vaccination budget is limited. If the same limited amount of vaccine is available, then giving priority to populations with higher mobility may reduce the total number of infections more effectively than distributing the vaccination effort uniformly. We also note that $\sigma$ in our model is a global parameter describing the mobility level of the population. Hence, this interpretation should be understood as adjusting the vaccination strategy based on the overall mobility profile of the population, not as identifying particular mobile individuals inside the population.

\section*{Acknowledgments}
This work was initiated during the author's doctoral studies at Khalifa University in 2024--2025 and formed part of the author's Ph.D. thesis submitted in July 2025. It was further developed and completed while the author was affiliated with Mohamed bin Zayed University of Artificial Intelligence. The author thanks Prof. Mokhtar Kirane and Prof. Eduardo Cuesta for comments and discussions related to an earlier version of this work, and Prof. Hadi Susanto for discussions and his continuous support that led to the finalisation of this article. The author further thanks Prof. Neil Trudinger for his warm hospitality and generous encouragement during the author's visit to the Australian National University, where this paper was finalised, and Prof. Jiakun Liu for the hospitality during a visit to the University of Sydney. The author also thanks Prof. Rowena Ball, Cale Rankin, Xiaoxu Wu, and the ANU Mathematical Sciences Institute for their guidance and warm welcome during the author's stay in Canberra.

\end{document}